\documentclass{amsart}
\usepackage{latexsym,amsxtra,amscd,ifthen}
\usepackage{amsfonts}
\usepackage{verbatim}
\usepackage{amsmath}
\usepackage{amsthm}
\usepackage{amssymb}

\numberwithin{equation}{section}

\theoremstyle{plain}
\newtheorem{theorem}{Theorem}[section]
\newtheorem{lemma}[theorem]{Lemma}

\newtheorem{proposition}[theorem]{Proposition}
\newtheorem{corollary}[theorem]{Corollary}

\theoremstyle{definition}
\newtheorem{definition}[theorem]{Definition}
\newtheorem{example}[theorem]{Example}

\newtheorem{convention}[theorem]{Convention}
\newtheorem{remark}[theorem]{Remark}
\newtheorem*{remark*}{Remark}
\newtheorem{question}[theorem]{Question}

\makeatletter              
\let\c@equation\c@theorem  
\makeatother

\DeclareMathOperator{\GL}{GL}

\newcommand{\fg}{\mathfrak g}

\DeclareMathOperator{\gldim}{gldim}

\DeclareMathOperator{\Ext}{Ext}

\DeclareMathOperator{\pdim}{projdim}
\DeclareMathOperator{\fdim}{flatdim}

\DeclareMathOperator{\gr}{gr}

\DeclareMathOperator{\injdim}{injdim}

\DeclareMathOperator{\GKdim}{GKdim}

\newcommand{\fm}{\mathfrak{m}}

\newcommand\nat{{\rm($\natural$)}}

\begin{document}

\title[Hopf algebras of GK-dimension two]
{Hopf algebras of GK-dimension two \\
with vanishing Ext-group}

\author{D.-G. Wang, J.J. Zhang and G. Zhuang}

\address{Wang: School of Mathematical Sciences,
Qufu Normal University, Qufu, Shandong 273165, P.R.China}

\email{dgwang@mail.qfnu.edu.cn, dingguo95@126.com}

\address{Zhang: Department of Mathematics, Box 354350,
University of Washington, Seattle, Washington 98195, USA}

\email{zhang@math.washington.edu}

\address{Zhuang: Department of Mathematics, Box 354350,
University of Washington, Seattle, Washington 98195, USA}

\email{gzhuang@math.washington.edu}

\begin{abstract}
We construct and study a family of finitely generated Hopf algebra 
domains $H$ of Gelfand-Kirillov dimension two such that 
$\Ext^1_H(k,k)=0$. Consequently, we answer a question of 
Goodearl and the second-named author.
\end{abstract}





\maketitle


\setcounter{section}{-1}
\section{Introduction}
\label{xxsec0} 
Analysis of Hopf algebras of low Gelfand-Kirillov dimension (or 
GK-dimension for short) is an important step in understanding basic 
properties and algebraic structures of general Hopf algebras. Noetherian 
prime regular Hopf algebras of GK-dimension one were studied in 
\cite{BZ, Li, LWZ}. The study of noetherian Hopf algebra domains 
(or Hopf domains, for short) of GK-dimension two was started by Goodearl 
and the second-named author \cite{GZ} in which a classification was 
obtained under the extra hypothesis \nat\; [Theorem \ref{xxthm1.4}]. 
Let $k$ be an algebraically closed field of characteristic zero. Recall 
from \cite{GZ} that \nat\; means the following non-vanishing condition 
of the first $\Ext$-group

{\it $\quad$ \nat $:$ \; $\Ext^1_H(_Hk,_Hk)\neq 0$, where $_Hk$ denotes 
the trivial left $H$-module.}

\noindent
A well-known example of Hopf domains of GK-dimension two is the 
quantized enveloping algebra of the positive Borel subalgebra of 
${\mathfrak {sl}}_2(k)$, which is isomorphic to $A(1,q):= 
k\langle x^{\pm1},y\rangle/(xy-qyx)$ where $q$ is a nonzero
scalar and $\Delta(x)=x\otimes x, \Delta(y)= y\otimes1+ x\otimes y$
(defined in Example \ref{xxex1.1}). One can check the condition \nat\; by
verifying $\Ext^1_{A(1,q)}(k,k)=\begin{cases} k & {\text{if}} \quad 
q\neq 1\\ k\oplus k & {\text{if}} \quad q=1\end{cases}$. It is natural 
to ask if every Hopf domain of GK-dimension two satisfies \nat, 
see \cite[Question 0.3]{GZ}. Our first goal is to construct finitely 
generated noetherian Hopf domains $H$ of GK-dimension two 
such that $\Ext^1_H(k,k)=0$, or equivalently, that the condition \nat\; 
fails. As a consequence, \cite[Question 0.3]{GZ} is answered negatively. 

We investigate a family of Hopf algebras of GK-dimension two with 
vanishing $\Ext^1_H(k,k)$, denoted by $K(\{p_i\},\{q_i\},\{\alpha_i\},M)$ 
(see Section 2). A subfamily of which, denoted by $B(n,\{p_i\}_{1}^s,q,
\{\alpha_i\}_{1}^s)$, is a modification of 
$B(n,p_0,p_1,\cdots,p_s,q)$ introduced in \cite{GZ}. We conjecture that 
these $B(n,\{p_i\}_{1}^s,q,\{\alpha_i\}_{1}^s)$ are the only 
pointed Hopf domains of GK-dimension two that are missing 
from the list given in \cite[Theorem 0.1]{GZ}. The second goal of 
the paper is to prove the following theorem which provides an evidence
to the conjecture. We say that $H$ satisfies the hypothesis $\Omega'$ if

\begin{enumerate}
\item[]
{\it $\Omega':$ \; $H$ does not contain $A(1,q)$ as Hopf subalgebra 
for any $q$ being either }
\item[]
{\it a primitive 5th or a primitive 7th root of unity.}
\end{enumerate}

\begin{theorem}
\label{xxthm0.1}
Let $H$ be a Hopf domain of GK-dimension two such that
it is finitely generated by grouplike and skew primitive elements 
as an algebra and that $\Ext^1_H(k,k)=0$. If $H$ satisfies $\Omega'$,
then $H$ is isomorphic to $B(n,\{p_i\}_{1}^s,q,\{\alpha_i\}_{1}^s)$
with $\alpha_i\neq \alpha_j$ for some distinct integers $i$ and $j$.
\end{theorem}

There are seven families of noetherian Hopf domains of GK-dimension two 
which satisfy $\Ext^1_H(k,k)\neq 0$ [Theorem \ref{xxthm1.4}]. As an 
immediate consequence, we have 

\begin{corollary}
\label{xxcor0.2}
Let $H$ be a Hopf domain of GK-dimension two such that it is finitely 
generated by grouplike and skew primitive elements. If $H$ satisfies 
$\Omega'$, then $H$ is isomorphic to either the algebra in Theorem 
\ref{xxthm0.1} or one of the algebras in the seven families listed 
in Theorem \ref{xxthm1.4}.
\end{corollary}

It is unknown whether all finitely generated Hopf domains of GK-dimension 
two are generated by grouplike and skew primitive elements. If it is 
affirmative, Corollary \ref{xxcor0.2} provides a classification 
of finitely generated Hopf domains of GK-dimension two. Some basic 
properties of $B(n,\{p_i\}_{1}^s,q,\{\alpha_i\}_{1}^s)$ are listed in the 
next theorem. 

\begin{theorem}
\label{xxthm0.3}
Let $H$ be the algebra $B(n,\{p_i\}_{1}^s,q,\{\alpha_i\}_{1}^s)$. Then
\begin{enumerate}
\item
$H$ is finitely generated over its affine center.
\item
$\injdim H=2$.
\item
$\gldim H<\infty$ if and only if $\gldim H=2$ if and only if
$s=2$ and $\alpha_1\neq \alpha_2$.
\end{enumerate}
\end{theorem}

Although many statements hold over arbitrary base field $k$, we assume 
that $k$ is algebraically closed of characteristic zero for simplicity. 
All vector spaces, algebras, tensor products and linear maps are taken 
over $k$. Usually $H$ denotes a Hopf algebra over $k$. Our basic 
reference for Hopf algebras is the book \cite{Mo} and we denote counit, 
coproduct, and antipode by the symbols $\epsilon$, $\Delta$, and $S$, 
respectively.

\subsection*{Acknowledgments}
The authors thank Nicol{\'a}s Andruskiewitsch for Remark \ref{xxrem4.3} 
and thank Ken Goodearl, Martin Lorenz and Don Passman for Proposition 
\ref{xxprop1.5} and their proofs. J.J. Zhang thanks Ken Brown and Ken 
Goodearl for many valuable conversations on the subject during the last 
few years. A part of research was done when J.J. Zhang was visiting 
Fudan University in Fall quarter of 2009, Spring quarter of 2010 
and Fall quarter of 2010. J.J. Zhang and G. Zhuang were supported 
by the US National Science Foundation. 

\section{Review and some classification results}
\label{xxsec1}
This section is divided into two parts. The first part is a review of 
the work on a classification of Hopf domains of 
GK-dimension two under the condition \nat. The second part concerns 
a classification of pointed Hopf domains of GK-dimension two with 
$\GKdim C_0\neq 1$ where $\{C_i\}$ denotes the coradical filtration of 
$H$.

\subsection{Goodearl-Zhang's work}
\label{xxsec1.1}

We collect some examples of Hopf algebras of GK-dimension two and 
state the main result of \cite{GZ}. Everything in this subsection is 
from \cite{GZ}. A nonzero element $y\in H$ is {\it skew 
primitive}, or more precisely, {\it $(1,g)$-primitive}, if 
$$\Delta(y)=y\otimes 1+ g\otimes y$$
where $g$ is a grouplike element in $H$. Such a $g$ (uniquely determined 
by $y$) is called the {\it weight} of $y$ and denoted by $\mu(y)$. Note 
that $(g-1)$ is always a skew primitive element of weight $g$.
A skew primitive is called {\it trivial} if it is of the form
$c(g-1)$ for $c\in k^{\times}:=k\setminus \{0\}$ and for a 
grouplike element $g$. In most cases, a skew primitive element
is meant to be nontrivial. 

\begin{example}
\label{xxex1.1} Let $n\in{\mathbb Z}$ and $q\in
k^{\times}$, and set $A= k\langle x^{\pm1},y \mid
xy=qyx\rangle$. There is a unique Hopf algebra structure on $A$
under which $x$ is grouplike and $y$ is skew primitive, with
$\Delta(y)= y\otimes1+ x^n\otimes y$. This Hopf algebra is denoted
by $A(n,q)$. By \cite[Construction 1.1]{GZ}, if $m\in{\mathbb Z}$
and $r\in k^{\times}$, then $A(m,r)\cong A(n,q)$ if and only if
either $(m,r)= (n,q)$ or $(m,r)= (-n,q^{-1})$.
\end{example}

\begin{example}
\label{xxex1.2} Let $n,p_0,p_1,\dots,p_s$ be positive integers and
$q\in k^{\times}$ with the following properties:
\begin{enumerate}
\item $s\ge2$ and $1<p_1< p_2< \cdots< p_s$;
\item $p_0\mid n$ and $p_0,p_1,\dots,p_s$ are pairwise relatively prime;
\item $q$ is a primitive $\ell$-th root of unity, where $\ell=
(n/p_0)p_1p_2\cdots p_s$.
\end{enumerate}
Set $m=p_1p_2\cdots p_s$ and $m_i= m/p_i$ for $i=1,\dots,s$. Choose an
indeterminate $y$, and consider the subalgebra $A= k[y_1,\dots,y_s]
\subset k[y]$ where $y_i= y^{m_i}$ for $i=1,\dots,s$. The $k$-algebra
automorphism of $k[y]$ sending $y\mapsto qy$ restricts to an automorphism
$\sigma$ of $A$. There is a unique Hopf algebra structure on
the skew Laurent polynomial ring $B= A[x^{\pm1};\sigma]$ such that $x$ is
grouplike and the $y_i$ are skew primitive, with $\Delta(y_i)=
y_i\otimes1 + x^{m_in}\otimes y_i$ for $i=1,\dots,s$.
The Hopf algebra $B$ has GK-dimension two. We shall denote
it $B(n,p_0,\dots,p_s,q)$ \cite[Construction 1.2]{GZ}.
\end{example}

\begin{example}
\label{xxex1.3} Let $n$ be a positive integer and set
$C=k[y^{\pm1}] \bigl[ x; (y^n-y)\frac{d}{dy} \bigr]$. There is a
unique Hopf algebra structure on $C$ such that $\Delta(y)=y\otimes y$
and $\Delta(x)= x\otimes y^{n-1}+ 1\otimes
x$. This Hopf algebra is denoted by $C(n)$. For $m,n\in{\mathbb
Z}_{>0}$, the Hopf algebras $C(m)$ and $C(n)$ are isomorphic if and
only if $m=n$ \cite[Construction 1.4]{GZ}.
\end{example}

Here is the main result of \cite{GZ}. An algebra is called
{\it affine} if it is finitely generated as an algebra over
$k$. The condition \nat\; is defined in the introduction.

\begin{theorem}
\label{xxthm1.4} \cite[Theorem 0.1]{GZ}
Let $H$ be a Hopf domain of GK-dimension two 
satisfying \nat. 
Then $H$ is noetherian if and only if $H$ is affine, 
if and only if $H$ is isomorphic to one of the following:
\begin{enumerate}
\item[(I)]
The group algebra $k\Gamma$, where $\Gamma$ is either
\begin{enumerate}
\item[(Ia)]
the free abelian group ${\mathbb Z}^2$, or
\item[(Ib)]
the nontrivial semidirect product ${\mathbb Z}\rtimes{\mathbb Z}$.
\end{enumerate}
\item[(II)]
The enveloping algebra $U({\mathfrak g})$, where ${\mathfrak g}$ is
either
\begin{enumerate}
\item[(IIa)]
the $2$-dimensional abelian Lie algebra over $k$, or
\item[(IIb)]
the Lie algebra over $k$ with basis $\{x,y\}$ and $[x,y]=y$.
\end{enumerate}
\item[(III)]
The Hopf algebras $A(n,q)$ from Example \ref{xxex1.1}, for $n\ge0$.
\item[(IV)]
The Hopf algebras $B(n,p_0,\dots,p_s,q)$ from Example \ref{xxex1.2}.
\item[(V)]
The Hopf algebras $C(n)$ from Construction from Example
\ref{xxex1.3}, for $n\ge2$.
\end{enumerate}
Aside from the cases $A(0,q)\cong A(0,q^{-1})$, the Hopf algebras
listed above are pairwise non-isomorphic. 
\end{theorem}

It would be convenient if every noetherian Hopf algebra domain of 
GK-dimension two satisfied \nat, but the algebras defined in Section 2 
are counterexamples. 

\subsection{Partial results on pointed Hopf domains of GK-dimension two}
\label{xxsec1.2}
In this subsection we start a classification of pointed Hopf 
domains $H$ of GK-dimension strictly less than three. Note 
that we do not assume that $H$ satisfies the condition \nat\;
in this subsection. Since $H$ is pointed, the 
coradical $C_0$ of $H$ is a group algebra $kG$ where $G$ consists 
of all grouplike elements in $H$. Since $kG$ is a subalgebra of $H$, 
$\GKdim kG<3$. Then $\GKdim kG$ is either 0, or 1, or 2. We consider 
these three subcases.

\subsubsection{$\GKdim C_0=2$}
The following proposition was proposed by Goodearl and the 
second-named author and proved by Goodearl, Passman and 
Lorenz. We thank them for sharing their proofs with us. 

Recall that the {\it centeralizer} of an element $g$ in a group 
$\Gamma$ is defined to be
$${\bf C}_\Gamma(g)=\{ h\in \Gamma\;|\; hg=gh\}.$$
The centralizer of a subset in $\Gamma$ is defined similarly.
The {\it finite conjugate center} of a group $\Gamma$ is defined to be
\cite[p. 115]{Pa}
$$\Delta(\Gamma)=\{x\in \Gamma\;|\; [\Gamma:{\bf C}_\Gamma(x)]<\infty\}.$$
Please do not confuse $\Delta(\Gamma)$ with the coproduct $\Delta$ and 
with conventions $\Delta({\mathcal B}(V))$ and $\Delta^{+}({\mathcal B}(V))$
introduced and used locally in Section \ref{xxsec4}. 

\begin{proposition}[Goodearl-Zhang]
\label{xxprop1.5} 
If the group algebra $k\Gamma$ is an affine domain
of GK-dimension two, then $\Gamma$ is either ${\mathbb Z}^2$
or the nontrivial semidirect product ${\mathbb Z}\rtimes
{\mathbb Z}=\langle x,y\;| xy=y^{-1}x\rangle$. Namely, $k\Gamma$
is in Theorem \ref{xxthm1.4}(I). 
\end{proposition}

The following nice proof is due to Lorenz.

\begin{proof}[Proof of Proposition \ref{xxprop1.5} (Lorenz)]
First of all, since $k\Gamma$ is an affine domain, $\Gamma$ is
finitely generated and torsionfree. Since $\GKdim (k\Gamma)=2$, 
$\Gamma$ is abelian-by-finite with Hirsch number $2$. 

Let $A=\Delta(\Gamma)$ be the finite conjugate center of $\Gamma$. 
By \cite[Lemma 1.6, p. 117]{Pa} $A$ is abelian, and hence $A\cong 
{\mathbb Z}^2$. In this case $A$ is also the largest abelian 
subgroup of $\Gamma$. Since $A$ is abelian and $\Gamma$ is 
abelian-by-finite, we have $A={\bf C}_{\Gamma}(A)$
and $G:=\Gamma/A$ is finite. If $G=\{1\}$ 
then $\Gamma={\mathbb Z}^2$. So it remains to consider that case that 
$G\neq \{1\}$. Let $f$ be the homomorphism
$$G\hookrightarrow \GL(A) \longrightarrow \{\pm 1\}$$
defined by
$$g\longmapsto g_A\longmapsto \det g_A$$
where $g_A\in \GL(A)$ is given by the conjugation of $g$ on $A$.
We claim that $f$ is an isomorphism, or equivalently, $\det g_A=-1$
for all $1\neq g\in G$. Write $g=\gamma A$ with 
$\gamma\not\in A$. Since $G$ is finite, there is an $n$ such that
$g^n=1$, or $\gamma^n\in A$. But $\gamma^n\neq 1$ since $\Gamma$
is torsionfree. So $g_A$ has a nontrivial fixed point 
$\gamma^n\in A$. This implies that $g_A$ has Jordon canonical 
form $\begin{pmatrix} 1&0\\0 & \det g_A\end{pmatrix}$, and 
hence we must have $\det g_A=-1$. This proves the claim. 
Pick any $x\in \Gamma$ whose image is $g$ that generates $G$.
Note that $x^2\in Z(\Gamma)$, the center of $\Gamma$. Since $x_A(=g_A)$ 
has eigenvalues $1$ and $-1$, $Z(\Gamma)$ has rank 1, or equivalently, 
is infinite cyclic. Hence the subgroup $\langle x, Z(\Gamma)\rangle$ 
of $\Gamma$ is infinite cyclic too. Without loss of generality we 
assume that $\langle x, Z(\Gamma)\rangle$ is generated by $x$. 
Consequently, $Z(\Gamma)$ is generated by $x^2$. Furthermore, 
$A/Z(\Gamma)$ is infinite cyclic as well, with a generator $y$.
Since $x_A$ has the eigenvalue $-1$ in $y$, we have 
$$xyx^{-1}=y^{-1}x^{2r}$$
for some $r$. Replacing $y$ by $yx^{2s}$ for a suitable $s$, we may
assume that $r=0$ or $r=-1$. If $r=-1$, then $(xy)^2=1$, which
contradicts to the fact $\Gamma$ is torsionfree. Thus we must
have $r=0$ and $xyx^{-1}=y^{-1}$. Since $\Gamma$ is
generated by $x$ and $y$, $\Gamma=\langle x,y\;|\; xy=y^{-1}x\}
={\mathbb Z}\rtimes {\mathbb Z}$.
\end{proof}

\begin{lemma}
\label{xxlem1.6}
Let $H$ be a pointed Hopf domain. If
$\GKdim H<\GKdim C_0+1$, then $H=C_0$.
\end{lemma}

\begin{proof} Suppose on contrary that $H\neq C_0$. Then there is a 
nonzero skew primitive element $y\in C_1\setminus C_0$ such that 
$\Delta(y)=y\otimes 1+g\otimes y$ and $g^{-1}yg=\lambda y+\tau(g-1)$ 
for some $\lambda, \tau \in k$ \cite[Lemma 2.5]{WZZ1}. By 
\cite[Theorem 0.2]{WZZ1}, the hypothesis that $\GKdim H<\GKdim C_0+1$
implies that $\lambda$ is a $p$th primitive root of unity for 
some $p\geq 2$. Since $\lambda\neq 1$, we may assume 
$g^{-1}yg-\lambda y=0$ by \cite[Lemma 2.5]{WZZ1}. Since the Hopf
subalgebra $K$ generated by $g^{\pm 1}$ and $y$ is a noncommutative domain,
$\GKdim K\geq 2$ by \cite[Lemma 4.5]{GZ}, whence $\GKdim K=2$ and $K$ 
is isomorphic to $A(1, \lambda)$ defined in Example \ref{xxex1.1}. As 
a consequence, $y^{p}$ is a nontrivial skew primitive element. Since 
$g^{-p} y^{p} g^{p}=y^{p}$, \cite[Theorem 0.2]{WZZ1} implies that 
$\GKdim H\geq \GKdim C_0+1$, a contradiction.
\end{proof}

\begin{theorem}
\label{xxthm1.7}
Let $H$ be an affine pointed Hopf domain of 
GK-dimension strictly less than three. If the coradical $C_0$ is of  
GK-dimension two, then $H$ is isomorphic to $k\Gamma$ where $\Gamma$ 
is either ${\mathbb Z}^2$ or ${\mathbb Z}\rtimes {\mathbb Z}$ as given 
in Theorem \ref{xxthm1.4}(I). 
\end{theorem}

\begin{proof} The assertion follows from Lemma \ref{xxlem1.6} and 
Proposition \ref{xxprop1.5}.
\end{proof}

\subsubsection{$\GKdim C_0=0$}
In this case $C_0=k$ and $H$ is a connected Hopf algebra. 
Part (a) of the  following lemma is due to Le Bruyn (unpublished).

\begin{lemma} 
\label{xxlem1.8} Let $H$ be a connected Hopf algebra and 
$K$ be the associated graded Hopf algebra $\gr_C H$ with 
respect to the coradical of $H$.
\begin{enumerate}
\item
$H$ is a domain.
\item
$K$ is a connected graded Hopf algebra that is a domain
with $\GKdim K\leq \GKdim H$.
\item
Let $\fm$ be the graded maximal ideal of $K$. Then 
$\gr_\fm K$ is a universal enveloping algebra $U(\fg)$
where ${\mathfrak g}$ is a graded Lie algebra generated in degree one with
dimension no more than $\GKdim K$.
\end{enumerate}
\end{lemma}

\begin{proof} 
(c) This is a consequence of \cite[Proposition 3.4(a)]{GZ}. 

(b) Since $U(\fg)$ is a domain (for any $\fg$), so is $\gr_\fm K$. 
By definition, $K$ is a connected ${\mathbb N}$-graded Hopf algebra. 
Then $\bigcap_i \fm_K=0$, and so the filtration $\{\fm_K^i\}_{i\geq 0}$
is exhaustive and separated. Therefore $K$ is a domain.
By \cite[Lemma 6.5]{KL}, $\GKdim K\leq \GKdim H$. 

(a) This is a consequence of (b). 
\end{proof}

\begin{theorem}
\label{xxthm1.9}
Let $H$ be a connected Hopf domain of GK-dimension strictly
less than three. Then $H$ is isomorphic to $U(\mathfrak h)$ for a 
Lie algebra ${\mathfrak h}$ of dimension no more than $2$. If
$\GKdim H\geq 2$, then $H$ is isomorphic to Hopf algebras in Theorem 
\ref{xxthm1.4}(II). 
\end{theorem}

\begin{proof} To avoid triviality we assume that $\GKdim H\geq 2$.
By \cite[Theorem 1.1]{Zh}, $H$ contains a Hopf subalgebra
$U(\mathfrak h)$ for some 2-dimensional Lie algebra $\mathfrak h$.

Retain the notation from Lemma \ref{xxlem1.8}, we have 
$\gr_{\fm} K=U(\mathfrak g)$ and 
$$\dim {\mathfrak g}\geq \dim {\mathfrak g}_1\geq \dim K_1
=\dim C_1/C_0\geq \dim {\mathfrak h}=2.$$
By Lemma \ref{xxlem1.8}(b,c), $\dim {\mathfrak g}<3$, 
whence $\dim {\mathfrak g}=\dim {\mathfrak g}_1=\dim K_1=2$. Since 
${\mathfrak g}$ is graded and generated in degree 1, it must be 
a 2-dimensional abelian Lie algebra. Therefore $\gr_{\fm}(K)=
U(\mathfrak g)=U({\mathfrak g}_1)$ is commutative and generated 
by ${\mathfrak g}_1$. Since $K$ is connected graded and since 
$\dim K_1=2$, $K_1={\mathfrak g}_1=\fm/\fm^2$. This implies that 
$K\cong \gr_{\fm}(K)=U(\mathfrak g)$. As a consequence, $H$ is 
generated by primitive elements. The assertion follows from 
\cite[Theorem 5.6.5]{Mo}.
\end{proof}

\subsubsection{$\GKdim C_0=1$} 
\label{xxsec1.2.3}
The most difficult case is when $\GKdim C_0=1$. Since we assume 
$H$ is a domain, so is $C_0$. Thus $C_0$ is commutative 
\cite[Lemma 4.5]{GZ}. By \cite[Proposition 2.1]{GZ},
$C_0=kG$ where $G$ is abelian of rank one. 

\begin{lemma}
\label{xxlem1.10} Let $H$ be a pointed Hopf algebra that is 
finitely generated as algebra over $k$. Then $C_0$ is 
finitely generated.
\end{lemma}

\begin{proof} 
Suppose $H$ is generated by a finite dimensional subcoalgebra $V$. 
Then $V$ is a pointed coalgebra. Let $F$ be the free (pointed) Hopf 
algebra generated by $V$, which is defined in \cite{Ta2}. By the 
universal property of $F$, there is a Hopf algebra sujective map 
$F\to H$. It follows from \cite[Theorem 35]{Ta2} that the coradical 
$F_0$ of $F$ is generated by the coradical of $V$. Hence $F_0$ is 
finitely generated. By \cite[Corollary 5.3.5]{Mo}, $G(H)$ is a quotient
of $G(F)$. Therefore $G(H)$ is finitely generated and the assertion
follows.
\end{proof}

See Theorem \ref{xxthm6.2} for a result in this subcase. 

To conclude this section we state a result of \cite{WZZ2}. An algebra 
$A$ is called {\it PI} if it satisfies a polynomial identity and $A$ 
is called {\it locally PI} if every affine subalgebra of $A$ is PI.

\begin{theorem}
\label{xxthm1.11}\cite[Theorem 7.2]{WZZ2} 
Let $H$ be an affine pointed Hopf domain such that $\GKdim H<3$
and that $C_0=k{\mathbb Z}$. If $H$ is not PI, then $H$ is isomorphic 
to one of following
\begin{enumerate}
\item
The Hopf algebra $A(n,q)$ of Example \ref{xxex1.1} where $n>0$ and
$q$ is not a root of unity.
\item
The Hopf algebra $C(n)$ of Example \ref{xxex1.3} for $n\geq 2$.
\end{enumerate}
\end{theorem}

Note that the proof of \cite[Theorem 7.2]{WZZ2} does not use 
anything in this paper. Combining \cite[Theorem 7.2]{WZZ2} with 
\cite[Lemma 4.5]{GZ} and Theorems \ref{xxthm1.7} and \ref{xxthm1.9} 
we have the following Corollary. 

\begin{corollary}
\label{xxcor1.12}\cite[Theorem 0.1]{WZZ2} 
Let $H$ be an affine pointed Hopf domain of GKdimension strictly less 
than three. If $H$ is not PI, then $H$ is isomorphic to one of following
\begin{enumerate}
\item
The enveloping algebra $U(\mathfrak g)$ of 2-dimensional non-abelian
solvable Lie algebra ${\mathfrak g}$ as in Theorem \ref{xxthm1.4}(IIb).
\item
The Hopf algebra $A(n,q)$ of Example \ref{xxex1.1} where $n> 0$ and
$q$ is not a root of unity.
\item
The Hopf algebra $C(n)$ of Example \ref{xxex1.3} for $n\geq 2$.
\end{enumerate}
\end{corollary}

As a consequence of the above Corollary, there is no affine
pointed Hopf domain of GK-dimension strictly between 2 and 3
\cite[Corollary 7.3]{WZZ2}. 

\section{Definition and elementary properties}
\label{xxsec2}

By the last section only un-classified (and more interesting) affine pointed 
Hopf domains of GKdimension two are PI and satisfy $\GKdim C_0=1$ and
$\Ext^1_H(k,k)=0$, which will occupy our attention for the rest of the paper. 

In this section we construct and study our main object -- a class of 
Hopf domains with $\Ext^1_H(k,k)=0$. We first introduce a more 
general class, denoted by $K$, dependent on a set of parameters with 
various conditions listed below. Suppose 
\begin{enumerate}
\item[(I2.0.1)]
$s\ge2$ and $M\geq 2$ are two integers; 
\item[(I2.0.2)]
$n_1,\cdots,n_s, p_1,\cdots,p_s$ are positive integers
such that $M=n_i p_i$ for any $i$;
\item[(I2.0.3)] 
$q_1,\cdots, q_s$ are nonzero scalars in $k$;
\item[(I2.0.4)]
for each $i$, both $q_i$ and $q_i^{n_i}$ are primitive $p_i$-th 
roots of unity;
\item[(I2.0.5)]
$q_j^{n_i}=q_i^{-n_j}$ for all $i<j$;
\item[(I2.0.6)]
$\alpha_1,\cdots, \alpha_s$ are scalars in $k$.
\end{enumerate}
There are two more conditions to consider. We will see soon in Lemma 
\ref{xxlem2.3}(b,c) that the Hopf algebra $K$ is a domain if and 
only if 
\begin{enumerate}
\item[(I2.0.7)]
$\gcd(p_i,p_j)=1$ for all $i\neq j$,
\end{enumerate}
and that $K$ satisfies the vanishing condition $\Ext^1_K(k,k)=0$ if and 
only if 
\begin{enumerate}
\item[(I2.0.8)]
$\alpha_i\neq \alpha_j$ for some $i\neq j$.
\end{enumerate}
In the rest of this section we fix a parameter set 
$$\{s, M, \{n_i\}_{i=1}^s, 
\{p_i\}_{i=1}^s, \{q_i\}_{i=1}^s, \{\alpha_i\}_{i=1}^s\}$$ 
satisfying 
(I2.0.1-I2.0.6). Let $K$ be the algebra generated by $x^{\pm 1}, y_1,
\cdots,y_s$ subject to the following relations
\begin{enumerate}
\item[(I2.0.9)]
$x x^{-1}=x^{-1} x=1$,
\item[(I2.0.10)]
$y_i x =q_i x y_i $ for all $i$,
\item[(I2.0.11)]
$y_jy_i= q_{j}^{n_i} y_iy_j$ for all $i<j$,
\item[(I2.0.12)]
$y_j^{p_j}=y_i^{p_i}+(\alpha_j-\alpha_i)(x^{M}-1)$ for all $i<j$.
\end{enumerate}
It is easy to see that the parameters $\{\alpha_i\}_{i=1}^s$ can be replaced 
by $\{0,\alpha_2-\alpha_1,\cdots, \alpha_s-\alpha_1\}$ without changing
the algebra. In other words, we may assume that $\alpha_1=0$. 
If $p_j=1$ for some $j$, then relation (I2.0.12) says that $y_j$ is 
generated by $y_i$ and $x$, so we can remove $y_j$ from the generating 
set without changing the algebra $K$. By choosing a minimal generating 
set we may assume that 
\begin{enumerate}
\item[(I2.0.13)]
$p_i\geq 2$ for all $i$. 
\end{enumerate}

\begin{lemma}
\label{xxlem2.1} 
The algebra $K$ has a $k$-linear basis of monomials
$$\{x^{w_0} y_1^{w_1}y_2^{w_2}\cdots y_s^{w_s}\}$$
where $w_0\in {\mathbb Z}, w_1\in {\mathbb N}$ and $0\leq w_i\leq
p_i-1$ for all $2\leq i\leq s$. 
\end{lemma}

\begin{proof}
We use Bergman's Diamond Lemma \cite[Theorem 1.2]{Be}. Define a linear 
order on the set of generators as follows
$$x^{-1}<x<y_1<\cdots < y_s.$$
By using relations in (I2.0.9)-(I2.0.12) it is 
easy to see that the algebra 
$K$ is generated by $x^{\pm 1},y_1,\cdots,y_s$ subject to the following
relations, with leading monomials in the left-hand side of the equations,
\begin{enumerate}
\item[(I2.1.1)]
$x x^{-1}=1$,
\item[(I2.1.2)]
$x^{-1} x=1$,
\item[(I2.1.3)]
$y_i x=q_i x y_i $ for all $i$,
\item[(I2.1.4)]
$y_i x^{-1}=q_i^{-1} x^{-1} y_i $ for all $i$,
\item[(I2.1.5)]
$y_jy_i= q_{ij} y_iy_j$ for all $i<j$ where $q_{ij}=q_j^{n_i}=q_i^{-n_j}$,
\item[(I2.1.6)]
$y_j^{p_j}=y_1^{p_1}+\alpha_j (x^{M}-1)$ for all $j>1$ (where we assume
$\alpha_1=0$).
\end{enumerate}
Using these relations, every element in $K$ can be written as a linear
combination of monomials listed the assertion. Therefore the given set 
of monomials $\{x^{w_0} y_1^{w_1}y_2^{w_2}\cdots y_s^{w_s}\}$ span the 
algebra $K$. 

To prove these monomials form a basis, it suffices to show that all 
ambiguities generated by relations (I2.1.1-I2.1.6) can be resolved 
(see Diamond Lemma \cite[Theorem 1.2]{Be}). The rest of the proof
amounts to verifying the required statement.

The first ambiguity is created between (I2.1.1) and (I2.1.2), which 
can be resolved as follows.
$$(x x^{-1})x=1 x=x, \quad {\text{and}}\quad
x(x^{-1} x)=x 1=x.$$ 

To save space we only resolve two more
ambiguities. As noted before we may assume that $\alpha_1=0$.

The ambiguity between (I2.1.3) and (I2.1.6) is obtained from the monomial
$y_j^{p_j} x$. It is easy to see that 
$$y_j^{p_j-1} (y_j x)=q_j^{p_j} x (y_j^{p_j})=x(y_1^{p_1}+\alpha_j
(x^{M}-1))$$ 
and that
$$(y_j^{p_j}) x=(y_1^{p_1}+\alpha_j
(x^{M}-1))x=x(y_1^{p_1}+\alpha_j (x^{M}-1)).$$
So the ambiguity is resolved.

The ambiguity between (I2.1.5) and (I2.1.6) can be resolved as below.
For any $i<j$,
$$\begin{aligned}
y_j^{p_j-1} (y_jy_i)&= q_{ij}^{p_j} y_i (y_j^{p_j})=
y_i(y_1^{p_1}+\alpha_j(x^{M}-1))\\
&=(q_{1i}^{-p_i}y_1^{p_1}+\alpha_j(x^{M}-1))y_i\\
&=(y_1^{p_1}+\alpha_j(x^{M}-1))y_i
\end{aligned}
$$ and
$$(y_j^{p_j}) y_i= (y_1^{p_1}+\alpha_j(x^{M}-1))y_i.\qquad\quad $$
So the ambiguity is resolved.

It is routine to check that all other ambiguities can be resolved and
therefore the assertion follows.
\end{proof}

The coalgebra structure of $K$ is defined by the following rules
\begin{enumerate}
\item[(I2.1.7)]
$\Delta(x)=x\otimes x,\; \epsilon(x)=1$,
\item[(I2.1.8)]
$\Delta(x^{-1})=x^{-1}\otimes x^{-1}, \; \epsilon(x^{-1})=1$,
\item[(I2.1.9)]
$\Delta(y_i)=y_i\otimes 1+x^{n_i}\otimes y_i,\; \epsilon(y_i)=0$.
\end{enumerate}

\begin{lemma}
\label{xxlem2.2} 
The algebra $K$ is a Hopf algebra using the rules defined by 
{\rm{(I2.1.7)-(I2.1.9)}} 
and the antipode is determined by the following rules
\begin{enumerate}
\item[]
$S(x)=x^{-1},\; S(x^{-1})=x$,
\item[]
$S(y_i)=-x^{-n_i} y_i=-q_i^{n_i}y_i x^{-n_i}$ for all $i$.
\end{enumerate}
\end{lemma}

\begin{proof} It is easy to verify that rules (I2.1.7)-(I2.1.9)
define algebra homomorphisms $\Delta: K\to K\otimes K$ and $\epsilon:
K\to k$  since both of them maps relations of $K$ to zero. 
Coassociativity and counit axioms hold since these axioms hold for 
the generators. This proves that $K$ is a bialgebra.

Note that $S$ extends to an algebra anti-automorphism of $K$. To check
$K$ is a Hopf algebra we only need to apply the antipode axiom to
the generators, which can be verified directly.
\end{proof}

The Hopf algebra $K$ is denoted by $K(\{p_i\}, \{q_i\}, \{\alpha_i\}, M)$
if we need to indicate the parameters. Note that $n_i=M/p_i$ for all $i=
1,\cdots,s$.

\begin{lemma}
\label{xxlem2.3} 
Let $K$ be 
$K(\{p_i\}, \{q_i\}, \{\alpha_i\}, M)$.
\begin{enumerate}
\item
$\gcd(p_i, n_i)=1$ for all $i$. 
\item
If $K$ is a domain (or, more generally, $K$ has a quotient Hopf algebra 
domain $K'$ of GK-dimension two), then $q_j^{n_i}=q_i^{n_j}=1$ and 
$\gcd(p_i,p_j)=1$ for all $i\neq j$. As a consequence, {\rm{(I2.0.7)}} 
holds.
\item
$K$ satisfies \nat\; if and only if $\alpha_i=\alpha_j$ for all $i\neq j$.
\end{enumerate}
\end{lemma}

\begin{proof} (a) Since $q_i$ and $q_i^{n_i}$ are both primitive
$p_i$-th root of unity, $\gcd(p_i,n_i)=1$.

(b) First we assume $K$ is a domain. Fix any $i\neq j$. Since $k$ is 
algebraically closed, there is a $\gamma$ such that $\alpha_j-\alpha_i
=\gamma^{p_i}$. We re-write the relation 
$y_j^{p_j}=y_i^{p_i}+(\alpha_j-\alpha_i) (x^{M}-1)$ as 
$y_j^{p_j}=y_i^{p_i}+\gamma^{p_i}(x^{p_i n_i}-1)$.
Since $y_i x^{n_i} =q_i^{n_i}x^{n_i} y_i$ and $q_i^{n_i}$ is a primitive 
$p_i$-th root of unity, we have 
$y_j^{p_j}=(y_i+\gamma x^{n_i})^{p_i}-\gamma^{p_i}$.
The relation $y_jy_i=q_j^{n_i} y_i y_j$ implies that 
$y_j(y_i+\gamma x^{n_i})=q_j^{n_i} (y_i+\gamma x^{n_i}) y_j$.
Thus the subalgebra $Y$ generated by $a:=y_i+\gamma x^{n_i}$ and
$b:=y_j$ has GK-dimension at most one. To see this, note that 
$k\langle a,b\rangle/(ba-q_j^{n_i}ab)$ is a domain of GK-dimension two
and that $Y$ is a proper quotient of $k\langle a,b\rangle/(ba-q_j^{n_i}ab)$.
Since $K$ is a domain, so is $Y$. By \cite[Lemma 4.5]{GZ}, $Y$ is 
commutative, and whence, $q_j^{n_i}=q_i^{-n_j}=1$. Consequently, $p_j$ 
divides $n_i$. By part (a), $\gcd(p_i,p_j)=\gcd(p_i, n_i)=1$.

If $K$ has a quotient Hopf algebra $K'$ which is a domain of GK-dimension 
two, the proof can be modified so that $ab=ba$ in $K'$ where $a$ is
the image of $y_i+\gamma x^{n_i}$ and $b$ is the image of $y_j$ in $K'$. 
Further that $y_iy_j=q_j^{n_i} y_i y_j$ in $K$ implies that 
$ab=q_j^{n_i} ba$ in $K'$. If $q_j^{n_i}\neq 1$,
then either $a$ or $b$ is 0 in $K'$. In either cases, equation (I2.1.6)
implies that the image of $y_1^{p_1}$ is in coradical of $K'$, which 
is $C_0(K')=k[x,x^{-1}]$. Consequently, the image of $y_j^{p_j}$ is
also in $C_0(K')$ for all $j$. Therefore $\GKdim K'=1$, a contradiction.
Thus $q_j^{n_i}=1$, which leads to the conclusion. 

(c) If $\alpha_i=\alpha_j$ for all $i,j$, then $K/(y_1,\cdots,y_s)$ is 
isomorphic to $k[x,x^{-1}]$, which is an infinite dimensional
commutative Hopf algebra. Hence \nat\; holds following 
\cite[Theorem 3.8(c)]{GZ}. 

Suppose $\alpha_i\neq \alpha_i$ for some $i\neq j$. As noted before 
we may assume $p_i\geq 2$ for all $j$ to avoid triviality. 
Then $q_i$ is not 1 for each $i$. Relation (I2.0.10) implies that
$y_i\in [K,K]$ for all $i$. Thus $K/[K,K]$ is isomorphic to 
$k[x,x^{-1}]/(x^{M}-1)$ by relation (I2.0.12), which is finite dimensional.
By \cite[Theorem 3.8(c)]{GZ}, \nat\; fails.
\end{proof}

\begin{proposition}
\label{xxprop2.4} 
The following are equivalent.
\begin{enumerate}
\item
$K(\{p_i\}, \{q_i\}, \{\alpha_i\}, M)$ is a domain.
\item
$\gcd(p_i,p_j)=1$ for all $i\neq j$. (Consequently, $p_j\mid n_i$ for all
$i\neq j$). 
\item
There exists a nonzero scalar $q$ such that $q_i=q^{m_i}$ for each $i$ 
where $m_i=(p_1\cdots p_s)/p_i$, and in this case 
$K(\{p_i\}, \{q_i\}, \{0\}, M)$ is isomorphic to a domain 
$B(n,n,p_1,\cdots, p_s,q)$ for $n:=M/(p_1\cdots p_s)$.
\item
$K(\{p_i\}, \{q_i\}, \{0\}, M)$ is a domain.
\end{enumerate}
\end{proposition}

\begin{proof} (a) $\Rightarrow$ (b) This is Lemma \ref{xxlem2.3}(b).

(b) $\Rightarrow$ (c) We consider the
algebra $K:=K(\{p_i\}, \{q_i\}, \{0\}, M)$ under the
hypothesis that $\gcd(p_i,p_j)=1$ for all $i\neq j$. 

Since $M=p_i n_i=p_j n_j$ and $\gcd(p_j,p_i)=1$,
$p_j\mid n_i$ for all $i\neq j$. Then $q_j^{n_i}=1$ and consequently,
$y_i$ commutes with $y_j$. Let $A$ be the subalgebra of $K$ generated 
by $y_1,\cdots, y_s$. Then we have relations
$$y_iy_j=y_jy_i,\quad y_i^{p_i}=y_j^{p_j}$$
for all $i\neq j$. By Lemma \ref{xxlem2.1}, there is no
other relations in $Y$. By the proof of \cite[Construction 1.2]{GZ}, 
$Y$ is isomorphic to a subalgebra of $k[y,y^{-1}]$ by identifying $y_i$ 
with $y^{m_i}$. Further $YS^{-1}$ is isomorphic to $k[y,y^{-1}]$
where $S$ is the set of all monomials in $y_1,\cdots,y_s$. Using the
relations (I2.0.9)-(I2.0.12) it is easy to see that $S$ is an Ore set 
of $K$ and 
the localization $KS^{-1}$ equals $k[y,y^{-1}][x,x^{-1};\sigma]$
where $\sigma$ is a graded algebra automorphism of $k[y,y^{-1}]$.
Let $q$ be the scalar such that $yx=q xy$. Then 
$$y_i x=y^{m_i} x=q^{m_i} xy^{m_i}=q^{m_i} xy_i$$
for all $i$. By comparing the above equation with (I2.0.10), we obtain
that $q_i=q^{m_i}$. Recall that we assume $\alpha_i=0$ for all $i$. In this
case the algebra $K(\{p_i\}, \{q_i\}, \{0\}, M)$ is exactly the algebra 
$B(n,n,p_1,\cdots, p_s,q)$ in Example \ref{xxex1.2} (with $p_0=n$). 
The assertion follows and $K$ is a domain. 

(c) $\Rightarrow$ (d) Clear.

(d) $\Rightarrow$ (a) Define an ${\mathbb N}$-filtration on 
$K(\{p_i\}, \{q_i\}, \{\alpha_i\}, M)$ by setting $\deg(x)=0$,
$\deg y_i=m_i$. Then the associated graded ring is isomorphic to
$K(\{p_i\}, \{q_i\}, \{0\}, M)$. Since $K(\{p_i\}, \{q_i\}, \{0\}, M)$
is a domain, so is $K(\{p_i\}, \{q_i\}, \{\alpha_i\}, M)$.
\end{proof}

By Proposition \ref{xxprop2.4}, the Hopf algebra 
$K(\{10,15\}, \{q^3, q^{-2}\},\{0,1\}, 30)$ is not a domain where 
$q$ is a primitive $30$th root of unity. The choice of $\{q_s\}$ 
is not unique even if all other parameters are fixed. For example, 
$K(\{10,15\}, \{q^3, q^{8}\},\{0,1\}, 30)$ is also a Hopf algebra
of the same kind. 

\begin{convention}
\label{xxcon2.5} Suppose (I2.0.1)-(I2.0.7) hold for the parameter set
used for the algebra $K$. Re-arranging $\{p_i\}_{i=1}^s$ we may 
assume that $1<p_1<p_2<\cdots< p_s$. Let $\ell=p_1\cdots p_s$, 
$m_i=\ell/p_i$ and $n=M/\ell$.
By Proposition \ref{xxprop2.4} and Example \ref{xxex1.2}, there is an 
$\ell$-th root of unity $q$ such that $q_i=q^{m_i}$ for every $i$.
As a consequence, since $M/(p_ip_j)$ is an integer for $i<j$, 
$$q_j^{n_i}=(q^{m_j})^{n_i}=
q^{\frac{\ell M}{p_j p_i}}=(q^{\ell})^{\frac{M}{p_j p_i}}=1$$
for all $i<j$. In this case, 
the algebra $K(\{p_i\}, \{q_i\}, \{\alpha_i\}, M)$ is a domain and 
denoted by $B(n,\{p_i\}_{1}^s,q,\{\alpha_i\}_{1}^s)$. In other words, 
the algebra $B(n,\{p_i\}_{1}^s,q,\{\alpha_i\}_{1}^s)$ is generated 
by $x^{\pm 1}, y_1,\cdots, y_s$ and subject to the relations
\begin{enumerate}
\item[]
$\;$ $x x^{-1}=x^{-1} x=1$,
\item[]
$\;$
$y_i x =q^{m_i} x y_i $ for all $i$,
\item[]
$\;$
$y_jy_i= y_iy_j$ for all $i<j$,
\item[]
$\;$
$y_j^{p_j}=y_i^{p_i}+(\alpha_j-\alpha_i)(x^{M}-1)$ for all $i<j$,
\end{enumerate}
with comultiplication and counit determined by (I2.1.7)-(I2.1.9)
and antipode determined by rules in Lemma \ref{xxlem2.2}. If $\alpha_i
=\alpha_j$ for all $i,j$, then $B(n,\{p_i\}_{1}^s,q,\{\alpha_i\}_{1}^s)$ 
is just the algebra $B(n,n,p_1,\cdots,p_s,q)$. 
Note that $n=M/(p_1\cdots p_s)$ and that we have removed $p_0$ from the 
above $B$ convention since $p_0=n$. 
\end{convention}

We continue to work on the algebra $K(\{p_i\}, \{q_i\}, \{\alpha_i\}, M)$ 
without assuming (I2.0.7) although our main interest is about 
$B(n,\{p_i\}_{1}^s,q,\{\alpha_i\}_{1}^s)$.

\begin{lemma}
\label{xxlem2.6} 
The coalgebra structure of $K(\{p_s\}, \{q_s\}, \{\alpha_s\}, M)$ is 
independent of $\{\alpha_i\}$. In particular, 
$K(\{p_s\}, \{q_s\}, \{\alpha_s\}, M)$ is isomorphic to
$K(\{p_s\}, \{q_s\}, \{ 0\}, M)$ as coalgebras.
\end{lemma}

\begin{proof} We use the $k$-linear basis given in Lemma \ref{xxlem2.1}. 
Then the coproduct of $K(\{p_s\}, \{q_s\}, \{\alpha_s\}, M)$ and
the coproduct of $K(\{p_s\}, \{q_s\}, \{0\}, M)$ coincide.
Hence the assertion follows.
\end{proof}

By Lemma \ref{xxlem2.3}(c), $K(\{p_s\}, \{q_s\}, \{\alpha_s\}, M)$ (when
$\alpha_i\neq \alpha_j$) is not isomorphic to
$K(\{p_s\}, \{q_s\}, \{ 0\}, M)$ as algebras.

\begin{theorem}
\label{xxthm2.7} Let $K:=K(\{p_i\}, \{q_i\}, \{\alpha_i\}, M)$ be 
defined as above.
\begin{enumerate}
\item
The algebra $K$ is affine and noetherian.
\item
$K$ is pointed and the coradical of $K$ is $C_0=k[x,x^{-1}]\cong 
k{\mathbb Z}$.
\item
$K$ is finitely generated over its affine center.
\item
$\GKdim K=2$. 
\item
$\injdim K=2$.
\item
$\gldim K$ is finite if and only if $\gldim K=2$ if and only if $s=2$ 
and $\alpha_1\neq \alpha_2$. 
\end{enumerate}
\end{theorem}

\begin{proof}
(a) By definition, $K$ is affine. It is noetherian 
since $K$ is a factor ring of an iterated Ore extension 
$k[x^{\pm 1}][y_1,\sigma_1]\cdots [y_s,\sigma_s]$ where automorphisms
$\sigma_i$ can be read off from relations (I2.0.9)-(I2.0.11).

(b) Using Lemma \ref{xxlem2.1} it can be checked directly that 
$C_0(K)=k[x,x^{-1}]$, so it is pointed.

(c) Let $Z$ be the center of $K$ and let $T$ be the subalgebra 
of $Z$ generated by central elements $y_1^{p_1}, x^{M}$ and $x^{-M}$.
By Lemma \ref{xxlem2.1} 
$$K=\sum_{0\leq w_0<M, 0\leq w_i< p_i, \forall i} T x^{w_0}
y_1^{w_1}\cdots y_s^{w_s}.$$ 
Hence $K$ is finitely generated over its center $Z$ and $Z$ is affine. 

(d) Since $T$ is isomorphic to $k[s,s^{-1}][t]$ which  has GK-dimension
two and since $K$ is finitely generated over $T$, $\GKdim K=2$. 

(e) This is a consequence of (a,d) and \cite[Theorem 0.1]{WZ}.

(f) Suppose $\alpha_i=\alpha_j$ for some $i\neq j$. Without
loss of generality, we may assume that $\alpha_1=\alpha_2$.
Let $K_0$ be the Hopf subalgebra generated by $x^{\pm 1},y_1$
and $y_2$. Then it follows from Lemma \ref{xxlem2.1} that $K$
is a left and a right free $K_0$-module. By \cite[Proposition 2.2(i)]{MR}
$$\pdim k_{K_0}\leq \pdim k_{K}$$
or equivalently, by \cite{LL}, 
$$\gldim K_0\leq \gldim K.$$
Let $Y$ be the subalgebra generated by $y_1$ and $y_2$. Then $K_0
=Y[x,x^{-1}]$. Since $Y$ is a connected graded domain of GK-dimension one
and it is not isomorphic to the polynomial ring $k[y]$,
$\gldim Y=\infty$. By \cite[Theorem 7.5.3(ii)]{MR},
$\gldim K_0=\infty$. Consequently, $\gldim K=\infty$.

Next we consider the case when $s\geq 3$ and $\alpha_i\neq \alpha_j$ 
for all $i\neq j$. We will deal with the case $s=2$ at the
end. Let $K_1$ be the Hopf subalgebra generated by $x^{\pm 1},y_1,y_2$
and $y_3$. It follows from Lemma \ref{xxlem2.1} that $K$
is a left and a right free $K_1$-module. The argument in the previous
paragraph shows that it suffices to show $\gldim K_1=\infty$. In other
words, we may assume $s=3$. Let $A$ be the subalgebra generated by
$y_1,y_2,y_3$ and $x^{\pm M}$. Note that relation (I2.0.12) 
implies that $x^M=1+(\alpha_i-\alpha_j)^{-1}(y_i^{p_i}-y_j^{p_j})$.
From this it is easy to check that
$A$ is isomorphic to $BS^{-1}$ where $B$ is a connected graded 
algebra generated by $y_1,y_2,y_3$ subject to the relations
\begin{enumerate}
\item[(i)]
$y_iy_j=q_j^{n_i} y_jy_i$ for all $i,j$,
\item[(ii)]
$(\alpha_1-\alpha_2)^{-1}(y_1^{p_1}-y_2^{p_2})=
(\alpha_2-\alpha_3)^{-1}(y_2^{p_2}-y_3^{p_3})=
(\alpha_1-\alpha_3)^{-1}(y_1^{p_1}-y_3^{p_3}).$
\end{enumerate}
And $S$ consists of all powers of the elements $\{1+
(\alpha_i-\alpha_j)^{-1}(y_i^{p_i}-y_j^{p_j}) \mid i\neq j\}$. 
Since $B$ is connected graded, PI of GK-dimension two and
since $B$ is not isomorphic to a skew polynomial ring,
$\gldim B=\infty$ \cite[Theorem 3.5]{StZ}
and $\fdim k_B=\pdim k_B=\infty$ where $k_B$ is the
module $B/B_{\geq 0}$. Since $k_BS^{-1}=k_A$ where $k_A$ is a 1-dimensional
right $A$-module. By \cite[Proposition 7.4.2(iii)]{MR}, 
$\fdim k_A=\infty$, so 
$$\gldim A\geq \pdim k_A=\fdim k_A=\infty.$$ 
Finally note that $K$ is a free $A$-module 
$K=\oplus_{i=0}^{nm-1} x^{i} A$ by Lemma 
\ref{xxlem2.1}. Thus $\pdim k_K\geq \pdim k_A=\infty$.

The remaining case is when $s=2$ and $\alpha_1\neq \alpha_2$.
After a scalar change we may assume that $\alpha_1=0$ and 
$\alpha_2=1$. So we have a relation
$$y_1^{p_1}-y_2^{p_2}=1-x^{M}.$$
Let $A$ be the algebra $k_{q_2^{n_1}}[y_1,y_2]S^{-1}$ where $S$ 
consists of all powers of the element $\{1-y_1^{p_1}+y_2^{p_2}\}$. 
Then $A$ has global dimension $2$. The algebra $K$ is a direct 
sum $\oplus_{i=0}^{M-1}x^{i} A$ which is in fact a crossed product 
$A\star ({\mathbb Z}/(M))$. By \cite[Theorem 7.5.6(iii)]{MR},
$\gldim K=\gldim A=2$.
\end{proof}

Theorem \ref{xxthm0.3} is a consequence of Theorem \ref{xxthm2.7}.
Part (f) of Theorem \ref{xxthm2.7} suggests the following questions.

\begin{question}
\label{xxque2.8}
Suppose $H$ is a noetherian affine Hopf algebra of GK-dimension
$n$. 
\begin{enumerate}
\item
Is there a function $f$ of $n$ such that the global 
dimension of $H$ is either infinite or bounded by $f(n)$?
\item
Assume $H$ is a domain (or a prime algebra) with finite 
global dimension. Is there a function $f$ of $n$ such that the 
minimal number of generators of $H$ is bounded by $f(n)$?
\end{enumerate}
\end{question}

\begin{lemma}
\label{xxlem2.9} Let $K$ be as in Theorem \ref{xxthm2.7}.
\begin{enumerate}
\item
If $y$ is a $(1,g)$-primitive elements not in $C_0$, then either
$g=x^{n_i}$ or $g=x^M$ and $y$ is a linear combination of 
$\{y_1,\cdots,y_s,y_1^{M}\}$ modulo $C_0$.
Consequently, the set $\{n_1,\cdots,n_s,M\}$ is an invariant of $K$.
\item
Suppose that $n_1,\cdots,n_s$ are distinct. Then every 
Hopf automorphism $f$ of $K$ is of the form $\phi: x\to x, 
y_i\to c_i y_i$, for all $i$, where $c_i\in k^{\times}$ satisfies 
$c_i^{p_i}=c_j^{p_j}$ for all $i,j$ and $c_i^{p_i}=1$ for all $i$ when 
$\alpha_k\neq \alpha_l$ for some $k,l$.
\item
Suppose $K$ is a domain.
If $K'=K(\{p'_i\}, \{q'_i\}, \{\alpha'_i\}, M')$ is another
algebra and $f: K\to K'$ is a Hopf surjective map. Then $f$ is 
an isomorphism. Up to a permutation of $\{1,2,\cdots,s\}$, there is 
a scalar $c\in k^{\times }$ such that $p'_i=p_i$, $q'_i=q_i$, $\alpha'_i=
c \alpha_i$ for all $i$.
\end{enumerate}
\end{lemma}

\begin{proof}
(a) We use induction on $s$. When $s=0$, the statement is trivial. 
Suppose now $s\geq 1$ and assume that the assertion holds for 
$K_{s-1}$ where $K_{s-1}$ is the Hopf subalgebra generated by 
$x^{\pm 1}, y_1,\cdots,y_{s-1}$. Let $F$ be a skew $(1,g)$-primitive 
element in $K$ but not in $K_{s-1}$ where $g$ is a grouplike
element. Write $F=\sum_{i=0}^{w}f_i y_s^{i}$ where $f_i\in K_{s-1}$ 
and $w<p_s$ and $f_w\neq 0$. If $w>1$, then 
$$\begin{aligned}
\Delta(F)&=\Delta(f_w y_s^{w}+\sum_{i=0}^{w-1} f_i y_s^{i})\\
&=\Delta(f_w)(y_s^w\otimes 1+ \sum_{j=1}^{w-1} 
{w\choose j}_{q_s} (x^{n_s j} y_s^j \otimes y_s^{w-j})
+x^{n_sw}\otimes y_s^w)+\Delta(\sum_{i=0}^{w-1} f_i y_s^{i}),\\
{\text{and   }} &{\text {since $F$ is $(1,g)$-primitive, }}\\
\Delta(F)&=F\otimes 1+ g\otimes F=(\sum_{i=0}^{w}f_i y_s^{i})\otimes 1+
g\otimes (\sum_{i=0}^{w}f_i y_s^{i}).
\end{aligned}
$$
Since $K$ is free over $K_{s-1}$ with basis $1,y_s,\cdots,y_s^{p_s}$,
we have ${w\choose 1}_{q_s} \Delta(f_w)(x^{n_s}\otimes 1)=0$
by comparing the coefficient of the term $y_s\otimes y_s^{w-1}$.
Since ${w\choose 1}_{q_s}(x^{n_s}\otimes 1)$ is invertible,
$\Delta(f_w)=0$. Consequently, $f_w=0$, yielding a contradiction. 
Therefore $w=1$ and $F=f_0+f_1 y_s$. Then 
$\Delta(F)=F\otimes 1+ g\otimes F$ implies that
$$\begin{aligned}
\Delta(f_1)&=f_1\otimes 1,\\
\Delta(f_1) (x^{n_s}\otimes 1)&=g\otimes f_1,\\
\Delta(f_0)&=f_0\otimes 1+g\otimes f_0.
\end{aligned}
$$
These equations imply that $f_1\in k$, $g=x^{n_s}$ and $f_0$ is a
$(1,g)$-primitive. The assertion follows by induction.

(b) Let $\phi$ be any Hopf automorphism of $K$. Since $\{n_i\}$
are distinct, it follows from part (a) that every 
$(1,x^{n_i})$-primitive element is of the form 
$c y_i+d (x^{n_i}-1)$ for some $c,d\in k$. Hence $\phi$ sends
$y_i$ to $c_i y_i+d_i (x^{n_i}-1)$ for $i=1,\cdots,s$ and
sends $y_1^{p_1}$ to $c y_1^{p_1}+d (x^M-1)$. The equation
$$c y_1^{p_1}+d (x^M-1)=\phi(y_1^{p_1})=\phi(y_1)^{p_1}=
(c_1 y_1+d_1 (x^{n_1}-1))^{p_1}$$
implies that $d=d_1=0$ and $c=c_1^{p_1}$. Similar, $d_i=0$
and $c_i^{p_i}=c$ for all $i$. Applying $\phi$ to $\Delta(y_i)$
we see that $\phi(x^{n_i})=x^{n_i}$, which implies that 
$\phi$ is an identity on $C_0$. If $\alpha_k\neq \alpha_l$ for 
some $k<l$, then $c_i^p=c=1$ by the equation (I2.1.6).
The assertion follows.

(c) When $K$ is a domain, condition (I2.0.7) implies
that $\{n_i\}$ are distinct. 
If $f$ is surjective, then $f$ is also injective since
$K$ is a domain and $\GKdim K=\GKdim K'=2$. Hence $f$ is 
an isomorphism. Similar to the proof of (b), one sees that
$f$ sends $x$ to $x$, $y_i$ to $c_i y_i$ up to a permutation. 
Thus $\{p_i\}=\{p'_i\}$, $\{q_i\}=\{q'_i\}$, $M=M'$, and 
$c_i^{p_i}=c_j^{p_j}$ for all $i,j$.
We may assume that $\alpha_1=\alpha'_1=0$. Then $\alpha'_j=c
\alpha_j$ where $c=c_j^{p_j}=c_i^{p_i}$ for all $i,j$.
\end{proof}

\section{Preliminary analysis on skew primitive elements}
\label{xxsec3}

It remains to prove Theorem \ref{xxthm0.1} and Corollary \ref{xxcor0.2}
in the rest of the paper. A basic idea is to analyze skew primitive
elements in more details. The analysis sometimes is tedious
but necessary, and will be useful for the study of pointed
Hopf algebras of GK-dimension three or higher. 

Let $H$ be a pointed Hopf algebra and let $C_0$ denote the coradical 
of $H$. We will need to use some concepts introduced in \cite{WZZ1}. 
Suppose $y$ is a skew primitive element with $\Delta(y)=y\otimes 1+
g\otimes y$ where the grouplike element $g(=\mu(y))$ is the weight 
of $y$. Let $T_{g^{-1}}$ denote the inverse conjugation by $g$ 
sending $a\to g^{-1} ag$ for all $a\in H$. A nonzero scalar $\lambda$ 
is called the {\it commutator} of $y$ of level $n$ (or more 
generally of finite level) if 
$$(T_{g^{-1}}-\lambda Id_H)^n(y)\in C_0, \quad {\text{ and}}\quad 
(T_{g^{-1}}-\lambda Id_H)^{n-1}(y)\not\in C_0.$$  
If the commutator of $y$ of finite level exists, it is denoted by 
$\gamma(y)$. When $n=1$, $\gamma(y)$ is called the {\it commutator} 
of $y$. The weight commutator of $y$ is the pair $(\mu(y),\gamma(y))$
which is denoted by $\omega(y)$. 

Given a grouplike element $g$, let $P_{g,*,*}$ denote the span of
all $(1,g)$-primitive elements in $H$. It is clear that $P_{g,*,*}
\cap C_0=k(g-1)$. Let $P'_{g,*,*}$ denote the quotient space 
$P_{g,*,*}/k(g-1)$. Given a nonzero scalar $\lambda$ and an integer 
$n$, let $P_{g,\lambda, n}$ denote the span of all $(1,g)$-primitive 
elements having commutator $\lambda$ of level at most $n$. Let 
$P_{g,\lambda,*}$ be the union of $P_{g,\lambda, n}$ for all $n$. 
Similarly let $P'_{g,\lambda, n}$ (respectively, $P'_{g,\lambda,*}$)
denote $P_{g,\lambda, n}/k(g-1)$ (respectively, $P_{g,\lambda,*}/k(g-1)$).
Note that $P_{g,\lambda, 0}=k(g-1)$ and whence $P'_{g,\lambda, 0}=0$. 
The total space of nontrivial skew primitive elements is defined to be
\begin{equation}
\label{I3.0.1}\tag{I3.0.1}
P'_T:=\bigoplus_{g} P'_{g,*,*}=\bigoplus_{g}\bigoplus_{\lambda}
P'_{g,\lambda,*}
\end{equation}
(see Lemma \ref{xxlem3.2}(b)). 
Let $R_n$ be the set of primitive $n$th roots of unity in $k$ and 
let $\sqrt{\;}$ denote $\bigcup_{n\geq 2} R_n$ -- the set of all 
roots of unity in $k$ which is not 1. Let 
$$P'_M:=\bigoplus_{g} \bigoplus_{\lambda \not\in \sqrt{\;}} 
\; P'_{g,\lambda, *}.$$

\begin{definition} 
\label{xxdefn3.1}
\begin{enumerate}
\item
If $y\in P_{g,\lambda,*}\setminus C_0$ for some $\lambda\not\in 
\sqrt{\;}$, then $y$ is called a  {\it major skew primitive 
element}, $g$ is called a {\it major weight} of $H$ and $\lambda$ 
a {\it major commutator} of $H$. 
\item
If $P'_M$ is 1-dimensional, then there is only one grouplike element
$g$ and only one scalar $\gamma\not\in \sqrt{\;}$ such that
$$P'_M=P'_{g,\lambda,*}=P'_{g,\lambda,1}.$$ 
In this case the major weight $g$ is unique and the major 
commutator $\lambda$ is also unique. For simplicity, we say $H$ has a 
unique major skew primitive element in this case. 
\end{enumerate}
\end{definition}

By definition, when $H$ has a unique major skew primitive element, 
two major skew primitive elements are 
linearly dependent in the quotient space $P'_M$.

The major weights play a special role connecting non-major weights.

\begin{lemma}
\label{xxlem3.2} Suppose $\GKdim H<\infty$.
\begin{enumerate}
\item
$P_{g,*,*}$ is a sum of $P_{g,\lambda,*}$ for all $\lambda\in k^{\times}$. 
\item
$P'_{g,*,*}$ is a direct sum of $P'_{g,\lambda,*}$ for all 
$\lambda\in k^{\times}$.
\item
$y\in P_{g,\lambda, 1}$ means that $y$ is $(1,g)$-primitive and 
$g^{-1} y g= \lambda y+\tau(g-1)$ for some $\tau\in k$. 
\item
If $P_{g,\lambda, n}=P_{g,\lambda, n+1}$ for some $n$, then 
$P_{g,\lambda, n}=P_{g,\lambda, *}$. A similar statement holds when $P$
is replaced by $P'$.
\end{enumerate}
\end{lemma}

\begin{proof} (a) This is \cite[Lemma 3.7(a)]{WZZ1}, see also its proof.

(b) If $y\in P_{g,\lambda,*}\cap \sum_{\lambda'\neq \lambda}
P_{g,\lambda',*}$, then $y\in k(g-1)$. Hence $\sum_{\lambda }
P'_{g,\lambda,*}$ is a direct sum. The assertion follows from part (a).

(c,d) Easy.
\end{proof}

\begin{lemma}
\label{xxlem3.3} In parts (b) and (c) suppose that the coradical $C_0$ 
of $H$ is commutative and that $\GKdim H< \GKdim C_0+2<\infty$.
\begin{enumerate}
\item
Suppose $y$ is a nontrivial skew primitive element in 
Hopf domain $H$ such that $\gamma(y)$ is a primitive 
$p$th root of unity for some $p>1$. Then $y_0:=y^p$ is a major 
skew primitive element (after choosing $y$ properly). As a 
consequence, $\mu(y)^p$ is a major weight and $1$ is a major 
commutator and $\mu(y_0)^{-1} y_0\mu(y_0) =y_0$.
\item
The dimension of $P'_M$ is at most one. 
As a consequence, there is at most one pair $(g,\lambda)$ with 
$\lambda\not\in \sqrt{\;}$ such that $\dim P'_{g,\lambda, *}\neq 0$;
and in this case, $\dim P'_{g,\lambda, *}=\dim P'_{g,\lambda, 1}=1$.
\item
If $H$ is a pointed Hopf domain, then the dimension  of $P'_M$ is 1 
and the major weight and the major commutator are unique. Consequently,
there is exactly one pair $(g,\lambda)$ with $\lambda\not\in 
\sqrt{\;}$ such that $\dim P'_{g,\lambda, *}\neq 0$; and further,
$\dim P'_{g,\lambda, *}=\dim P'_{g,\lambda, 1}=1$.
\end{enumerate}
\end{lemma}

\begin{proof} 
(a) By Lemma \ref{xxlem3.2}(d), $P_{g,\lambda,1}\neq P_{g,\lambda,0}$
where $(g,\lambda)$ denotes $(\mu(y),\gamma(y))$. By choosing a 
different $y$ if necessary, we have $y\in P_{g,\lambda,1}\setminus C_0$. 
By Lemma \ref{xxlem3.2}(c), there is a $\tau\in k$ such that
$$g^{-1}y g=\lambda y+\tau(g-1).$$
Since $\lambda\neq 1$, we may assume $\tau=0$ after replacing $y$ by 
$y+{\frac{\tau}{1-\lambda}}(g-1)$. Let $H'$ be the Hopf subalgebra 
generated by $g^{\pm 1}$ and $y$. Since $H$ is a domain, so is $H'$. Since 
$H'$ is noncommutative (as $g^{-1}yg=\lambda y$), $\GKdim H'\geq 2$ 
by \cite[Lemma 4.5]{GZ}. It is easy to compute that $\GKdim H'\leq 2$. 
Then $H'$ is isomorphic to $A(1,\lambda)$ defined in Example 
\ref{xxex1.1}. Since $\lambda$ is a primitive $p$th root for some $p>1$, 
$y^p$ is a skew primitive with $\mu(y^p)=g^p$ and $\gamma(y^p)=1$. In 
the Hopf algebra $H'(\cong A(1,\lambda))$, $y^p$ is not in its coradical, 
or equivalently, $y^p\not\in k(g^p-1)$. Hence $y^p$ is not in the 
coradical of $H$. Therefore $y^p$ is a major skew primitive element. 
The consequence is clear. 

(b) Let $Y_*$ be the $k$-linear space spanned by all skew primitive 
elements $y$ such that the commutator of $y$ is not in $\sqrt{\;}$. 
By \cite[Theorem 3.9]{WZZ1},
$$\dim Y_*/(Y_*\cap C_0)\leq \GKdim H-\GKdim C_0<2$$
where the last inequality is the hypothesis. 
Since $\dim Y_*/(Y_*\cap C_0)$ is an integer, it is at most 1.
It is easy to see that 
$$Y_*/(Y_*\cap C_0)\cong 
\bigoplus_{g} \bigoplus_{\lambda \not\in \sqrt{\;}} 
\; P'_{g,\lambda, *}=P'_M.$$
The assertions follow easily. 

(c) By part (b) it is suffices to show that $P'_M\neq 0$. Suppose 
that on contrary $P'_M=0$. Since $H$ is pointed and $C_0\neq H$, there 
is a nontrivial skew primitive element $y$ in some $P_{g,\lambda,*}$ 
where $g=\mu(y)$. Since $P'_M=0$, the commutator $\lambda$ is in 
$\sqrt{\;}$. By part (a), $y^p$ is a major skew primitive element 
which is not in $C_0$. So $P'_M\neq 0$, a contradiction.
\end{proof}

\begin{lemma} 
\label{xxlem3.4}
Let $A$ be a locally PI domain.
\begin{enumerate}
\item
$A$ is an Ore domain and the quotient division ring $Q(A)$
of $A$ is locally PI.
\item
For each nonzero scalar $\lambda$, there are no nonzero elements 
$\{g, \alpha, \beta\}$ in $A$ such that $\alpha g= \lambda 
g\alpha+\beta$ and $\beta g=\lambda g\beta$.
\end{enumerate}
\end{lemma}

\begin{proof} (a) This is well-known.

(b) By part (a) we may assume that $A$ is a division algebra.
First assume that $\lambda=1$. 
Let $f=\alpha \beta^{-1}$, then $fg=gf+1$ by using the fact 
$g\beta=\beta g$. Thus $Q(A)$ contains the first Weyl algebra which 
is not (locally) PI, yielding a contradiction.

Next assume that $\lambda\neq 1$.
If $\lambda$ is not a root of unity, then $A$ contains a copy of
$k_\lambda[g,\beta]$ (since every proper prime factor ring of 
$k_\lambda[g,\beta]$ has to kill either $g$ or $\beta$). Since 
$k_\lambda[g,\beta]$ is not PI, a contradiction. The last case is 
when $\lambda$ is a primitive $p$th root of unity for some $p>1$. 
Let $G=g^p$. Then we have $\alpha G=\lambda^p G \alpha+n \beta 
g^{p-1}=G\alpha +\beta'$ and $\beta' G=G\beta'$ where 
$\beta'=n \beta g^{p-1}$. The assertion follows from the case when 
$\lambda=1$.
\end{proof}

\begin{proposition} 
\label{xxprop3.5}
Let $H$ be a Hopf domain that is either locally PI or having 
$\GKdim K<3$. Then $P_{g,\lambda,*}=P_{g,\lambda,1}$ for all pairs 
$(g,\lambda)$.
\end{proposition}

\begin{proof} By Lemma \ref{xxlem3.2}(d) it suffices to show the 
assertion that $P_{g,\lambda,2}=P_{g,\lambda,1}$, which is equivalent 
to the following claim: for any scalars $\lambda,b,c\in k$, there is no 
triple $(g,y_1,y_2)$ with $y_1\in P_{g,\lambda,1}\setminus C_0$ and 
$y_2\in P_{g,\lambda,2}\setminus C_0$ such that $g^{-1}y_1 g=
\lambda y_1+b(g-1)$ and $g^{-1} y_2 g= \lambda y_2+ y_1+c(g-1)$. 

Next we prove the claim. If $\lambda=0$, then $y_1\in C_0$ which yields
a contradiction. Hence $\lambda\neq 0$. Without loss of generality, we 
may assume that $H$ is generated by $g^{\pm 1},y_1,y_2$, whence $H$ 
is affine and pointed. We consider two cases. The first case is when 
$H$ is (locally) PI. If $\lambda=1$ and $b\neq 0$, then the statement 
follows from Lemma \ref{xxlem3.4}(b) by taking $\alpha=y_1$ and 
$\beta=bg(g-1)$. If $\lambda=1$ and $b=0$, then the statement follows from 
Lemma \ref{xxlem3.4}(b) by taking $\alpha=y_2$ and $\beta=g(y_1+c(g-1))$.
If $\lambda\neq 1$, then $y_1$ and $y_2$ can be modified so that
$b=c=0$. Then the statement follows from Lemma \ref{xxlem3.4}(b) by 
taking $\alpha=y_2$ and $\beta=gy_1$. This finishes the first case.
The second case is when $H$ is not (locally) PI, and then $\GKdim H<3$.
By Corollary \ref{xxcor1.12}, all non-PI affine pointed Hopf algebras 
of $\GKdim <3$ are classified, namely, algebras (IIb), (III) and (V) 
in Theorem \ref{xxthm1.4}. It is easy to verify the statement for 
all these Hopf algebras.
\end{proof}

Under the hypotheses of Proposition \ref{xxprop3.5}, there is an 
improved version of \eqref{I3.0.1}
$$P'_{T}=\bigoplus_{g}\bigoplus_{\lambda} P'_{g,\lambda,1}.$$
Theorem \ref{xxthm3.7} below gives a bound for $\dim P'_{g,\lambda,1}$. 

\begin{lemma}
\label{xxlem3.6}
Suppose that $H$ is a Hopf domain such that $\dim P'_{g,1,1}\leq 1$ 
for all $g$. If $\lambda$ is a primitive $p$th root of unity for 
some $p\geq 2$, then $\dim P'_{g,\lambda,1}\leq 1$ for all $g$.
\end{lemma}

\begin{proof} Suppose by contrary that $\dim P'_{g,\lambda,1}\geq 2$. 
Then there are two linearly independent $(1,g)$-primitive elements 
$y_1$ and $y_2$ such that $g^{-1} y_1 g= \lambda y_1$ and 
$g^{-1} y_2 g= \lambda y_2$.
Let $a$ and $b$ be two noncommutative variables. Then 
$$(a+b)^p=a^p+M_1(a,b)+M_2(a,b)+\cdots +M_{p-1}(a,b)+b^p$$
where $M_i(a,b)$ denote the sum  of all noncommutative monomials 
of $a$ and $b$ with total $(a,b)$-degree $(i,p-i)$. For
example, 
$$M_1(a,b)=a b^{p-1}+b ab^{p-2}+b^2 a b^{p-3}+\cdots +b^{p-1}a.$$
Let $\xi$ be a scalar in $k$. Then 
$$(a+\xi b)^p=a^p+ \xi M_1(a,b)+\sum_{i=2}^p \xi^i M_i(a,b).$$
For any $\xi\in k$ let $y_{\xi}:=y_1+\xi y_2$. Then $y_{\xi}$ is
a $(1,g)$-primitive and $g^{-1}y_{\xi} g=\lambda y_{\xi}$.
By the proof of Lemma \ref{xxlem3.3}(a), $y_{\xi}^p$ is in  
$P_{g^p,1,,1}\setminus C_0$ for any $\xi$. Let $f=y_1^p$. Since
$\dim P'_{g^p,1,,1}\leq 1$ by hypotheses, $P_{g^p,1,,1}=
kf+ k(g^p-1)$. This implies that 
$$\sum_{i=0}^p \xi^i M_i(y_1,y_2)=y_{\xi}^p\in kf+ k(g^p-1).$$
Since $k$ is infinite, the above equation implies that 
$M_i(y_1,y_2)\in kf+ k(g^p-1)$ for all $i$.  
Let $M_i(y_1,y_2)=a_i f+b_i(g^p-1)$. Choose $\xi$
so that $\sum_{i=0}^p a_i\xi^i =0$. Then
$$y_{\xi}^p=\sum_{i=0}^p \xi^i M_i(y_1,y_2)=
(\sum_{i=0}^p a_i\xi^i)f+(\sum_{i=0}^p b_i\xi^i)(g^p-1)=
b(g^p-1)$$
for some $b\in k$. Write $b=-c^p$ for some $c\in k$. Then
$$y_{\xi}^p+ (cg)^p-c^p=0.$$
Since $g^{-1}y_{\xi}g=\lambda y_{\xi}$, the above equation
is equivalent to $(y_{\xi}+cg)^p-c^p=0$, or $\prod_{n=0}^{p-1}
(y_{\xi}+cg-\eta_n c)=0$ where $\eta_n=e^{\frac{2n i \pi}{p}}$.
Since $H$ is a domain, $y_1+\xi y_2+cg-c'=0$. Since $g^{-1}(y_1+\xi y_2)g
=\lambda (y_1+\xi y_2)$, $g^{-1}(cg-c')g=(cg-c')$ and $\lambda\neq 1$, 
we have $c'=c=0$. This contradicts the fact $y_1$ and $y_2$ are 
linearly independent. We finish the proof.
\end{proof}

\begin{theorem}
\label{xxthm3.7}
Suppose that $H$ is a pointed Hopf domain with a commutative coradical 
$C_0$ and that $\GKdim H<\GKdim C_0+2<\infty$. Then 
$\dim P'_{g,\lambda,1}\leq 1$ for all pairs $(g,\lambda)$.
\end{theorem}

\begin{proof} If $\lambda$ is either 1 or not a root of unity,
the assertion follows from Lemma \ref{xxlem3.3}(c). If $\lambda$ is 
a primitive $p$th root of unity for $p>1$, the assertion follows 
from Lemma \ref{xxlem3.6}.
\end{proof}

An important consequence of Theorem \ref{xxthm3.7} is the following.

\begin{lemma}
\label{xxlem3.8} Let $H$ be a Hopf algebra such that 
$P'_{g,\lambda,1}$ is 1-dimensional for some pair $(g,\lambda)$.
If $G_0$ is an abelian subgroup of grouplike elements and it contains
$g$, then there is a $z\in P_{g,\lambda,1}\setminus C_0$ such that
either
\begin{enumerate}
\item
$h^{-1} z h=\chi(h) z$ for some character $\chi: G_0\to k^{\times}$
(where $\chi(g)=\lambda)$, or
\item
$h^{-1} z h=z+\tau(h)(g-1)$ for some additive character $\tau: G_0\to k$
and $\lambda=1$.
\end{enumerate}
In part (a), $z$ is unique up to a scalar multiple. In part (b), 
$z$ is unique up to an addition of $k(g-1)$.
\end{lemma}

\begin{proof} Let $y$ be any element in $P_{g,\lambda,1}\setminus C_0$
and let $V=k(g-1)+\sum_{h\in G_0} k (h^{-1} y h)$. Since $G_0$ is
abelian, $h^{-1}yh\in P_{g,\lambda,1}\setminus C_0$. Hence $V\subset
P_{g,\lambda,1}$. Thus $\dim V\leq 
\dim P_{g,\lambda,1}=2$. The assertion follows from 
\cite[Lemma 2.2(c)]{WZZ1}.
\end{proof}

Finally we prove that the total space of skew primitive elements is
finite dimensional.

\begin{theorem}
\label{xxthm3.9}
Let $H$ be a pointed Hopf domain of $\GKdim <3$ and suppose that
$C_0=k{\mathbb Z}$. Then $P'_T$ is finite dimensional.
\end{theorem}

\begin{proof} By Proposition \ref{xxprop3.5} and 
Theorem \ref{xxthm3.7}, $\dim P'_{g,\lambda,*}=\dim P'_{g,\lambda,1}
\leq 1$ for any pair $(g,\lambda)$. It suffices to show that there 
are only finitely many pairs $(g,\lambda)$ such that 
$\dim P'_{g,\lambda,*}\neq 0$.
By Lemma \ref{xxlem3.3}(c), there is exactly one pair $(g,\lambda)$
such that $\lambda\not\in \sqrt{\;}$ and $\dim P'_{g,\lambda,*}
=1$. Denote this pair by $(x^M,\nu)$. 

If there is another pair $(g,\lambda)$ such that 
$\dim P'_{g,\lambda,*}\neq 0$, then $\lambda\in \sqrt{\;}$
by Lemma \ref{xxlem3.3}(b). Write $g=x^{n}$. 
Pick $y\in P_{g,\lambda,*}\setminus C_0$ and we may assume 
that $x^{-1} y x=q y$ by Lemma \ref{xxlem3.8}(a). Then 
$\lambda=q^n$  and it is a primitive $p$th root of unity
for some $p>1$. Thus $y^p\in P_{x^{np},1,1}\setminus C_0$
by the proof of Lemma \ref{xxlem3.3}(a). Thus $P'_{x^{np},1,1}\neq 0$ and
whence $(x^{np},1)=(x^M,\nu)$. Since $\lambda\neq 1$, $n\neq 0$. Thus 
$M\neq 0$ and $M=np$. Since $M$ is fixed and $p>1$, there are only 
finitely many choices for $n$, or equivalently, finitely many choices 
for $g=x^n$. For each fixed $g=x^n$, $\lambda$ is a primitive $p$th 
root of unity where $p=M/n$. Thus the possibilities for $\lambda$ are 
also finite. Therefore there are only finitely many choices for 
pairs $(g,\lambda)$ such that $\dim P'_{g,\lambda,*}\neq 0$.
\end{proof}

We have an easy corollary.

\begin{corollary} 
\label{xxcor3.10}
Suppose $H$ is a pointed Hopf domain of $\GKdim <3$. If $H$ is finitely 
generated by grouplike and skew primitive elements, then $P'_T$ is 
finite dimensional.
\end{corollary}

\begin{proof} If $\GKdim H<2$, then $\GKdim H\leq 1$ and 
such Hopf algebras are classified in \cite[Section 2]{GZ}.
The assertion is easy to check. Suppose now $\GKdim H\geq 2$.
Then all affine pointed Hopf algebras are 
classified except for that case when $\dim C_0=1$, see subsection 
\ref{xxsec1.2}. 
So assertions can be verified when $\GKdim C_0\neq 1$. The remaining
case is when $\GKdim C_0=1$. Since $H$ is a domain, so is $C_0$.
Then $C_0=k \Gamma$ for an abelian torsionfree group $\Gamma$ of 
rank 1 by \cite[Section 2]{GZ}. 
By Lemma \ref{xxlem1.10}, $\Gamma$ is finitely generated. Thus
$\Gamma={\mathbb Z}$. Now Theorem \ref{xxthm3.9} applies.
\end{proof}

\section{A result of Heckenberger}
\label{xxsec4}
We need to use a result of Heckenberger \cite{He} which 
concerns the classification of finite dimensional Nichols 
algebras of rank 2. Let $G$ be a finite abelian group and let 
$^G_G{\mathcal YD}$ be the 
Yetter-Drinfel'd category. Let $V$ be a Yetter-Drinfel'd 
module over $kG$ of dimension 2 with left $kG$-action 
denoted by $*$ and the left $kG$-coaction denoted by 
$\delta: V\to kG\otimes V$. Assume that $V$ is of diagonal
type, namely, there is a basis $\{v_1,v_2\}$ such that
\begin{equation}
\label{I4.0.1}\tag{I4.0.1}
\delta(v_i)=g_i\otimes v_i, \quad i=1,2
\end{equation}
for $g_i\in G$, and
\begin{equation}
\label{I4.0.2}\tag{I4.0.2}
g_i* v_j= q_{ij} v_j, \quad i,j \in \{1,2\} 
\end{equation}
where $q_{ij}\in k^{\times}$. The braiding on $V$ is determined by
$$\sigma(v_i\otimes v_j)=q_{ij} v_j\otimes v_i$$
for all $i,j\in \{1,2\}$. The Nichols algebra over $V$ is denoted
by ${\mathcal B}(V)$. Heckenberger worked out the precise
conditions on $\{q_{ij}\}$ such that ${\mathcal B}(V)$ is finite
dimensional. To quote Heckenberger's result we need to introduce
a few other notations. Following \cite[p.118]{He}, let 
$\Delta^{+}({\mathcal B}(V))$ be the set of degrees of the
(restricted) Poincar{\'e}-Birkhoff-Witt generators counted
with multiplicities. From this, $\dim {\mathcal B}(V)<\infty$
if and only if $\Delta^{+}({\mathcal B}(V))$ is finite. Based on
$\Delta^{+}({\mathcal B}(V))$, one can define $\Delta({\mathcal B}(V))$, 
a subgroupoid $W_{\chi,E}$, and an arithmetic root system 
$(\Delta({\mathcal B}(V)), \chi,E)$ (details are omitted).
When $(\Delta({\mathcal B}(V)), \chi,E)$ is an arithmetic root 
system, we implicitly assume that $W_{\chi,E}$ is full and finite. 
A very nice result of Heckenberger \cite[Theorem 3]{He} states that 
there is a one-to-one correspondence between finite 
$\Delta^{+}({\mathcal B}(V))$ and arithmetic root systems 
$(\Delta, \chi,E)$. Below is a re-statement of a part of a remarkable 
result \cite[Theorem 7]{He}. Recall that $R_n$ is the set of primitive 
$n$th roots of unity. 

\begin{lemma}
\label{xxlem4.1}\cite{He}
Let $V$ be a 2-dimensional Yetter-Drinfel'd module over $kG$ of 
diagonal type with structure coefficients $(q_{ij})_{2\times 2}$ 
defined in \eqref{I4.0.2}. Suppose ${\mathcal B}(V)$ is finite 
dimensional. Then, up to a permutation of $\{v_1,v_2\}$, one of the 
following is true.
\begin{enumerate}
\item[(1)]
$q_{12}q_{21}=1$.
\item[(2)]
$q_{12}q_{21}\neq 1$, $q_{12}q_{21}q_{22}=1$, and 
\begin{enumerate}
\item[(2.1)]
$q_{11}q_{12}q_{21}=1$ or 
\item[(2.2)]
$q_{11}=-1$, $q_{12}^2q_{21}^2\neq 1$ or
\item[(2.3)]
$q_{11}^2q_{12}q_{21}=1$ or
\item[(2.4)]
$q_{11}^3q_{12}q_{21}=1$, $q_{11}^2\neq 1$ or 
\item[(2.5)]
$q_{11}\in R_3$, $q_{12}^3 q_{21}^3\neq 1$ or
\item[(2.6)]
$q_{12}q_{21}\in R_{8}, q_{11}=(q_{12}q_{21})^2$ or
\item[(2.7)]
$q_{12}q_{21}\in R_{24}, q_{11}=(q_{12}q_{21})^6$ or
\item[(2.8)]
$q_{12}q_{21}\in R_{30}, q_{11}=(q_{12}q_{21})^{12}$.
\end{enumerate}
\item[(3)]
$q_{12}q_{21}\neq 1, q_{11}q_{12}q_{21}\neq 1,
q_{12}q_{21}q_{22}\neq 1, q_{22}=-1, q_{11}\in
R_2\cup R_3$, and 
\begin{enumerate}
\item[(3.1)]
$q_{11}=-1, q_{12}^2q_{21}^2\neq 1$ or
\item[(3.2)]
$q_{11}\in R_3, q_{12}q_{21}\in \{q_{11}, -q_{11}\}$ or 
\item[(3.3)]
$q_0:=q_{11}q_{12}q_{21}\in R_{12}, q_{11}=q_0^4$ or
\item[(3.4)]
$q_{12}q_{21}\in R_{12}, q_{11}=-(q_{12}q_{21})^2$ or
\item[(3.5)]
$q_{12}q_{21}\in R_{9}, q_{11}=(q_{12}q_{21})^{-3}$ or
\item[(3.6)]
$q_{12}q_{21}\in R_{24}, q_{11}=-(q_{12}q_{21})^{4}$ or
\item[(3.7)]
$q_{12}q_{21}\in R_{30}, q_{11}=-(q_{12}q_{21})^{5}$.
\end{enumerate}
\item[(4)]
$q_{12}q_{21}\neq 1, q_{11}q_{12}q_{21}\neq 1, q_{12}q_{21}q_{22}\neq 1,
q_{22}=-1, q_{11}\not\in R_2\cup R_3$, and
\begin{enumerate}
\item[(4.1)]
$q_{12}q_{21}=q_{11}^{-2}$ or
\item[(4.2)]
$q_{11}\in R_5\cup R_8\cup R_{12}\cup R_{14}\cup R_{20},\; 
q_{12}q_{21}=q_{11}^{-3}$ or 
\item[(4.3)]
$q_{11}\in R_{10}\cup R_{18}, \; q_{12}q_{21}=q_{11}^{-4}$ or
\item[(4.4)]
$q_{11}\in R_{14}\cup R_{24}, \; q_{12}q_{21}=q_{11}^{-5}$ or
\item[(4.5)]
$q_{12}q_{21}\in R_8, \; q_{11}=(q_{12}q_{21})^{-2}$ or 
\item[(4.6)]
$q_{12}q_{21}\in R_{12}, \; q_{11}=(q_{12}q_{21})^{-3}$ or 
\item[(4.7)]
$q_{12}q_{21}\in R_{20}, \; q_{11}=(q_{12}q_{21})^{-4}$ or
\item[(4.8)]
$q_{12}q_{21}\in R_{30}, \; q_{11}=(q_{12}q_{21})^{-6}$.
\end{enumerate}
\item[(5)]
$q_{12}q_{21}\neq 1, q_{11}q_{12}q_{21}\neq 1, q_{12}q_{21}q_{22}\neq 1,
q_{11}\neq -1, q_{22}\in R_{3}$ and
\begin{enumerate}
\item[(5.1)]
$q_0:=q_{11}q_{12}q_{21}\in R_{12},\; q_{11}=q_0^4, q_{22}=-q_0^2$ or
\item[(5.2)]
$q_{12}q_{21}\in R_{12}, \; q_{11}=q_{22}=-(q_{12}q_{21})^2$ or
\item[(5.3)]
$q_{12}q_{21}\in R_{24},\; q_{11}=(q_{12}q_{21})^{-6}, q_{22}
=(q_{12}q_{21})^{-8}$ or
\item[(5.4)]
$q_{11}\in R_{18},\; q_{12}q_{21}=q_{11}^{-2}, q_{22}=-q_{11}^3$ or
\item[(5.5)]
$q_{11}\in R_{30},\; q_{12}q_{21}=q_{11}^{-3}, q_{22}=-q_{11}^5$.
\end{enumerate}
\end{enumerate}
\end{lemma}

\begin{proof}
By definition, $\Delta^{+}({\mathcal B}(V))$ is finite if
and only if ${\mathcal B}(V)$ is finite dimensional, and by
\cite[Theorem 3]{He}, if and only if $(\Delta({\mathcal B}(V)), 
\chi,E)$ is an arithmetic root system. By the definition of an 
arithmetic root system, $W_{\chi,E}$ is full and finite. By 
\cite[Page 131]{He}, $W_{\chi,E}$ is full and finite if and only if, 
up to a permutation of $\{v_1,v_2\}$, one of the cases listed above 
is true (according to \cite[page 131]{He}, this statement is also 
equivalent to \cite[Theorem 7]{He}). The assertion follows.  
\end{proof}

\begin{proposition}
\label{xxprop4.2} 
Retains the hypotheses of Lemma \ref{xxlem4.1}. Assume that
\begin{enumerate}
\item[(a)] 
there are two scalars $q_1$ and $q_2$ and two positive integers
$n_1$ and $n_2$ such that $q_{ij}=q_j^{n_i}$ for all $i,j\in
\{1,2\}$, and
\item[(b)]
$\gcd(n_1,n_2)=1$.
\end{enumerate}
Let $\epsilon$ be 
a positive integer and let $p_1=n_2 \epsilon$ and $p_2=n_1 \epsilon$.
Further assume that 
\begin{enumerate}
\item[(c)]
both $q_1$ and $q_{11}$ are primitive $p_1$st roots of unity, and
\item[(d)]
both $q_2$ and $q_{22}$ are primitive $p_2$nd roots of unity.
\end{enumerate}
Then, up to a permutation of $\{v_1,v_2\}$, one of the following 
holds.
\begin{enumerate}
\item[(I)]
$q_{12}q_{21}=1$,
\item[(II)]
$n_1=n_2=1$, $q_1, q_2\in R_3$.
\item[(III)]
$n_1=n_2=1$, $q_1, q_2\in R_5$.
\item[(IV)]
$n_1=1, n_2=2$, $\epsilon=5$, $p_1=10$, $p_2=5$ and
$q_1^4q_2=1$ and $q_1^2q_2^3=1$.
\item[(V)]
$n_1=n_2=1$, $\epsilon=7$, $q_1, q_2\in R_7$, $q_1q_2^2=1$ and
$q_1^4 q_2=1$.
\item[(VI)]
$n_1=1, n_2=3$, $\epsilon=7$, $p_1=21$, $p_2=7$, 
$q_1^3 q_2^4=1$ and that $q_1^6 q_2=1$. 
\end{enumerate}
\end{proposition}

\begin{proof} The proof is heavily dependent on Lemma \ref{xxlem4.1}.
First of all, case (1) in  Lemma \ref{xxlem4.1} is just case (I) here. 
If $\epsilon=1$, then $q_{21}=q_1^{n_2}=
q_1^{p_1}=1$ and similarly, $q_{12}=1$. Thus case (I) occurs.

Secondly, if $\epsilon=2$ then $q_{21}=q_1^{n_2}=-1$ as 
$q_1$ is a primitive $(2n_2)$nd root of unity. Similarly,
$q_{12}=-1$. Hence $q_{12}q_{21}=1$ and case (I) occurs.

Thirdly, if $q_{11}=-1$ then $p_1=2$. Thus $\epsilon$ is either
1 or 2. By the last two paragraphs, (I) occurs. Similarly
if $q_{22}=-1$, then (I) occurs. Thus cases (2.2), (3.1-3.7), (4.1-4.8)
in  Lemma \ref{xxlem4.1} can not happen under the extra hypotheses of
Proposition \ref{xxprop4.2}. It remains to analyze cases (2.1), 
(2.3-2.8), (5.1-5.5). Below we are using the numbering in Lemma 
\ref{xxlem4.1}.

Case (2): The equation $q_{12}q_{21}q_{22}=1$ means that
$q_2^{n_1}q_1^{n_2}q_2^{n_2}=1$. Then 
$$
1= 1^{\epsilon}=(q_2^{n_1}q_1^{n_2}q_2^{n_2})^\epsilon
=q_2^{p_2}q_1^{p_1}(q_2^{n_2})^\epsilon=(q_{22})^\epsilon
$$
Thus by hypothesis (d), $p_2=\epsilon$ which implies that
$n_1=1$. Below are subcases. 

Case (2.1): The equation $q_{11}q_{12}q_{21}=1$ implies that
$n_2=1$ and $q_1 q_2^2=1=q_1 q_2^2$ where the second equation
follows from $q_{12}q_{21}q_{22}=1$. These equations implies
$q_1, q_2\in R_3$, and whence case (II) occurs.

Case (2.3): The equation $q_{11}^2q_{12}q_{21}=1$ implies that
$$1=1^\epsilon =(q_{11}^2q_{12}q_{21})^\epsilon=
(q_{11}^2)^\epsilon q_{2}^{n_1\epsilon}q_{1}^{n_2\epsilon}
=q_{11}^{2\epsilon}.$$
Hence $n_2 \epsilon =p_1$ divides $2\epsilon$. So we have 
two subcases: either $n_2=1$ or $n_2=2$.

When $n_2=1$, the equation $q_{12}q_{21}q_{22}=1$ is 
that $q_1 q_2^2=1$ and the equation $q_{11}^2q_{12}q_{21}=1$ 
is that $q_1^3 q_2=1$. It is easy to see that $q_1, q_2\in
R_5$. Then case (III) occurs.

When $n_2=2$, the equation $q_{12}q_{21}q_{22}=1$ is 
that $q_1^2 q_2^3=1$ and the equation $q_{11}^2q_{12}q_{21}=1$ 
is that $q_1^4 q_2=1$. It is easy to see that $q_1\in R_{10}, 
q_2\in R_5$. Consequently, $\epsilon=5$. Then case (IV) occurs.

Case (2.4): The equation $q_{11}^3q_{12}q_{21}=1$ implies that
$$1=1^\epsilon =(q_{11}^3q_{12}q_{21})^\epsilon=
(q_{11}^3)^\epsilon q_{2}^{n_1\epsilon}q_{1}^{n_2\epsilon}
=q_{11}^{3\epsilon}.$$
Hence $n_2 \epsilon =p_1$ divides $3\epsilon$. So we have 
two subcases: either $n_2=1$ or $n_2=3$.

When $n_2=1$, the equation $q_{12}q_{21}q_{22}=1$ is 
that $q_1 q_2^2=1$ and the equation $q_{11}^3q_{12}q_{21}=1$ 
is that $q_1^4 q_2=1$. It is easy to see that $q_1, q_2\in
R_7$. Then case (V) occurs.

When $n_2=3$, the equation $q_{12}q_{21}q_{22}=1$ is 
that $q_1^3 q_2^4=1$ and the equation $q_{11}^3q_{12}q_{21}=1$ 
is that $q_1^6 q_2=1$. Consequently, $q_2^7=1$. Thus 
$p_2=7=\epsilon$. So case (VI) occurs. 

Case (2.5): Since $q_{11}\in R_3$, $p_1=3$. This implies that
either $\epsilon=1$ or $\epsilon=3$. If $\epsilon=1$, then (I)
occurs by the first paragraph. So we may assume $\epsilon=3$
and whence $n_2=1$. Since $n_1=1$ (see the beginning of case (2)),
$p_2=3n_1=3$. This is case (II).

Case (2.6): Since $q_{11}=(q_{12}q_{21})^{2}$, 
$$1=1^\epsilon=(q_{12}q_{21})^{\epsilon}=((q_{12}q_{21})^{\epsilon})^2
=q_{11}^\epsilon$$
which implies that $p_1=\epsilon$ and $n_2=1$. Since $n_1=n_2=1$,
$q_{12}q_{21}\in R_8$ means that $q_1q_2\in R_8$.
The equation $q_{12}q_{21}q_{22}=1$, see  at the beginning of
case (2),  implies that
$q_1 q_2^2=1$ which is equivalent to $q_2^{-1}=q_1 q_2\in R_8$.
Thus $p_1=p_2=\epsilon =8$. This contradicts the fact 
$q_1=q_2^{-2}\in R_4$. 

Case (2.7): Since $q_{11}=(q_{12}q_{21})^{6}$, 
$$1=1^\epsilon=(q_{12}q_{21})^{\epsilon}=((q_{12}q_{21})^{\epsilon})^6
=q_{11}^\epsilon$$
which implies that $p_1=\epsilon$ and $n_2=1$. Since $n_1=n_2=1$,
$q_{12}q_{21}\in R_{24}$ means that $q_1q_2\in R_{24}$.
The equation $q_{12}q_{21}q_{22}=1$ given at the beginning of
case (2) implies that
$q_1 q_2^2=1$ which is equivalent to $q_2^{-1}=q_1 q_2\in R_{24}$.
Thus $p_1=p_2=\epsilon =24$. This contradicts the fact 
$q_1=q_2^{-2}\in R_{12}$. 

Case (2.8): Since $q_{11}=(q_{12}q_{21})^{12}$, 
$$1=1^\epsilon=(q_{12}q_{21})^{\epsilon}=((q_{12}q_{21})^{\epsilon})^{12}
=q_{11}^\epsilon$$
which implies that $p_1=\epsilon$ and $n_2=1$. Since $n_1=n_2=1$,
$q_{12}q_{21}\in R_{30}$ means that $q_1q_2\in R_{30}$.
The equation $q_{12}q_{21}q_{22}=1$ implies that
$q_1 q_2^2=1$ which is equivalent to $q_2^{-1}=q_1 q_2\in R_{30}$.
Thus $p_1=p_2=\epsilon =30$. This contradicts the fact 
$q_1=q_2^{-2}\in R_{15}$. 

Case (5): Since $q_{22}\in R_3$, we have $p_2=3$. Since $\epsilon
\mid p_2$, $\epsilon$ is either 1 or 3. If $\epsilon=1$,
then (I) occurs, so a contradiction, by the first condition in case
(5). Thus $\epsilon =3$ and consequently, $n_1=p_2/\epsilon=1$. 
Below are subcases. 

Case (5.1): Since $q_0=q_{11}q_{12}q_{21}=q_1 q_2 q_1^{n_2}=
q_1^{(1+n_2)}q_2\in R_{12}$ and $q_{12}q_{21}\in R_3$, we have 
$q_1\in R_{12}$. Thus $p_1=12$, $n_2=4$. The equation
$q_{11}=q_0^4$ becomes $q_1=(q_1^5 q_2)^4=q_1^{-4}q_2$, which
implies that $q_2=q_1^5$. This contradicts the facts that $q_2\in 
R_3$ and that $q_1\in R_{12}$.

Cases (5.2) and (5.3): Since $(q_{12}q_{21})^\epsilon=1$ and $\epsilon
=3$, then $q_{12}q_{21}$ can not be in $R_{12}$ or $R_{24}$.
A contradiction.

Case (5.4): Since $\epsilon=3$, by $q_{12}q_{21}=q_{11}^{-2}$, 
$$1=1^{3} =(q_{12}q_{21})^{3}=q_{11}^{-6}$$
which contradicts the fact $q_{11}\in R_{18}$. 

Case (5.5): Since $\epsilon=3$, by $q_{12}q_{21}=q_{11}^{-3}$, 
$$1=1^{3} =(q_{12}q_{21})^{3}=q_{11}^{-9}$$
which contradicts the fact $q_{11}\in R_{30}$. 

This finishes the proof.
\end{proof}

We are interested in the case when $G={\mathbb Z}/(M)$ for some 
integer $M$ with a generator $x$ and when $V=kv_1\oplus kv_2$ is 
a Yetter-Drinfel'd module over $kG$ of diagonal type such that, 
for $i=1$ and $2$,
\begin{equation}
\label{I4.2.1}\tag{I4.2.1}
\delta(v_i)=x^{n_i}\otimes v_i, \quad x * v_i=q_i v_i
\end{equation}
for some $n_1,n_2\in {\mathbb N}$ and $q_1,q_2\in k^{\times}$.

\begin{remark}
\label{xxrem4.3} 
Andruskiewitsch informed us that the Nichols algebra ${\mathcal B}(V)$ 
is finite dimensional if $V$ satisfies \eqref{I4.2.1} and  
$\{n_1,n_2,q_1,q_2\}$ satisfies one of the following conditions:
\begin{enumerate}
\item
$n_1=n_2=1$, $q_1\in R_5$ and $q_2=q_1^2$.
\item
$n_1=n_2=1$, $q_1\in R_7$ and $q_2=q_1^3$.
\item
$n_1=1$ and $n_2=2$, $q_1\in R_{10}$ and $q_2=q_1^6$.
\item
$n_1=1$ and $n_2=3$, $q_1\in R_{21}$ and $q_2=q_1^{15}$.
\end{enumerate}
\end{remark}

Recall that the weight commutator of a skew primitive element $y$ is
$$\omega(y):=(\mu(y),\gamma(y)).$$

\begin{definition}
\label{xxdefn4.4}
\begin{enumerate}
\item
Let $N_5$ denote any Hopf domain of GK-dimension two that
is generated by $x^{\pm 1},y_1,y_2$ with $\omega(y_i)=(x^{n_i}, q_i^{n_i})$
for $i=1,2$ such that
$$n_1=n_2=1, \quad q_1,q_2\in R_5, \quad q_2=q_1^2.$$
\item
Let $N_{10}$ denote any Hopf domain of GK-dimension two that
is generated by $x^{\pm 1},y_1,y_2$ with $\omega(y_i)=(x^{n_i}, q_i^{n_i})$
for $i=1,2$ such that
$$n_1=1, n_2=2, \quad q_1\in R_{10},q_2\in R_5, \quad q_2=q_1^6.$$
\item
Let $N_{7}$ denote any Hopf domain of GK-dimension two that
is generated by $x^{\pm 1},y_1,y_2$ with $\omega(y_i)=(x^{n_i}, q_i^{n_i})$
for $i=1,2$ such that
$$n_1=n_2=1, \quad q_1,q_2\in R_7, \quad q_2=q_1^3.$$
\item
Let $N_{21}$ denote any Hopf domain of GK-dimension two that
is generated by $x^{\pm 1},y_1,y_2$ with $\omega(y_i)=(x^{n_i}, q_i^{n_i})$
for $i=1,2$ such that
$$n_1=1, n_2=3, \quad q_1\in R_{21},q_2\in R_7, \quad q_2=q_1^{15}.$$
\item
Algebras $N_5, N_{10}, N_7, N_{21}$ (if exist) are called {\it 
supplementary Hopf algebras of GK-dimension two}.
\end{enumerate}
\end{definition}

\begin{remark}
\label{xxrem4.5} As far as we know, no example of supplementary Hopf 
algebras of GK-dimension two has been constructed. This leads to 
conjecture that algebras $N_5, N_7, N_{10}$ and $N_{21}$ do not exist. 
If this conjecture is true, then the hypothesis $\Omega'$ in Theorem 
\ref{xxthm0.1} and Corollary \ref{xxcor0.2} can be removed and 
the hypothesis $\Omega$ defined in Section 6 is vacuous.
\end{remark}

\section{Analysis in the case $s=2$}
\label{xxsec5}
The proof of  Theorem \ref{xxthm0.1} requires some further analysis 
of skew primitive elements. In this section we prove the following 
special case of Theorem \ref{xxthm0.1}.

\begin{theorem}
\label{xxthm5.1}
Let $H$ be a Hopf algebra satisfying the following conditions
\begin{enumerate}
\item
$H$ is a domain of GK-dimension two.
\item
its coradical $C_0$ is $k{\mathbb Z}$ with a generator $x$.
\item
$H$ is generated by $x^{\pm 1}$, $y_1$ and $y_2$ where
$y_i\in P_{(x^{n_i},\lambda_i,*)}$ for $i=1,2$, and 
$H$ is not equal to the subalgebra generated by $\{x^{\pm 1},y_1\}$,
or by $\{x^{\pm 1},y_2\}$.
\item
$\gcd(n_1,n_2)=1$.
\end{enumerate}
Then $H$ is isomorphic to either
\begin{enumerate}
\item[(1)]
$B(1,\{p_i\}_{1}^2,q,\{\alpha_i\}_{1}^2)$ defined in 
Convention \ref{xxcon2.5} where $p_1=n_2,p_2=n_1$, and in this case,
$y_1y_2=y_2y_1$, 
\item[(2)]
or one of $N_5$, $N_7$, $N_{10}$ and $N_{21}$.
\end{enumerate}
\end{theorem}

The proof will be given at the end of the section and we start with 
the following lemma.

\begin{lemma}
\label{xxlem5.2} Retain the hypotheses of Theorem \ref{xxthm5.1}.
Then, after choosing $y_1,y_2$ appropriately, the following hold.
\begin{enumerate}
\item[(a)]
There are two scalars $q_1,q_2$ such that $y_i x= q_i xy_i$
and $\lambda_i= q_i^{n_i}$. 
\item[(b)]
$q_1$ and $q_1^{n_1}(=\lambda_1)$ are both primitive $p_1$st root 
of unity for some $p_1>1$. And $y_1^{p_1}$ is a major skew primitive
element.  
\item[(c)]
$q_2$ and $q_2^{n_2}(=\lambda_2)$ are both primitive $p_2$nd root 
of unity for some $p_2>1$. And $y_2^{p_2}$ is a major skew primitive
element.  
\item[(d)]
Replacing $x$ by $x^{-1}$ if necessary we may assume that $n_1\geq 0$.
Under this hypothesis, both $n_1$ and $n_2$ are positive integers
and $n_1p_1=n_2p_2$. The major weight is $x^{n_1p_1}$. 
\item[(e)]
There is a positive integer $\epsilon$ such that $p_1=n_2\epsilon$
and $p_2=n_1\epsilon$.
\item[(f)]
$y_1^{p_1}$ and $x^{n_1p_1}$ are central elements in $H$. As a consequence,
$H$ is PI.
\item[(g)]
$y_1^{p_1}$ is nonzero in the quotient Hopf algebra $H':=H/(x^{n_1p_1}-1)$.
\item[(h)]
$H/(y_1^{p_1}, (x^{n_1p_1}-1))$ is a finite dimensional Hopf
algebra and the image of $y_1$ (respectively, the image of $y_2$) 
is nonzero in $H/(y_1^{p_1},x^{n_1p_1}-1)$.
\item[(i)]
Let $y$ is a skew primitive element in $H\setminus C_0$ with 
$\omega(y)=(x^{n_3}, \lambda_3)$ with $\gcd(n_3,n_1)=1$. 
Then there is a scalar $q_3$ and 
positive integers $n_3$ and $p_3$ such that 
\begin{enumerate}
\item[(i1)] 
$\lambda_3= q_3^{n_3}$, 
\item[(i2)] both $q_3$ and $\lambda_3$ are primitive $p_3$rd root
of unity, and 
\item[(i3)] $n_3 p_3=n_1p_1=n_2 p_2$.
\end{enumerate}
\end{enumerate}
\end{lemma}

\begin{proof} (a,d) By hypothesis (c) of Theorem \ref{xxthm5.1}, $y_1$ and 
$y_2$ are linearly independent in $H/C_0$. By Theorem \ref{xxthm3.7}, 
$(x^{n_1},\lambda_1)\neq (x^{n_2},\lambda_2)$. By Lemma \ref{xxlem3.3}(b), 
there is at least one $\lambda_i$ which is in $\sqrt{\;}$. By symmetry, 
we may assume that $\lambda_1$ is a primitive $p_1$st root of unity
for some $p_1>1$, and further we assume that $n_1>0$ without
loss of generality. Since $\lambda_1\neq 1$, Lemma \ref{xxlem3.8}(b)
can not happen, so by Lemma \ref{xxlem3.8}(a), there is a $z\in 
P_{(x^{n_1},\lambda_1,1)}\setminus C_0$ such that $h^{-1} z h=\chi(h)
z$ for all $h\in {\mathbb Z}$. Replacing $y_1$ by $z$, we may assume 
that $y_1 x= q_{1}  x y_1$ where $\lambda_1=\chi(x^{n_1})=
(\chi(x))^{n_1}=q_1^{n_1}$. This also says that $q_1^{n_1}$ is a primitive 
$p_1$st root of unity. Therefore $y_1^{p_1}\in P_{x^{n_1p_1}, 1,1}$ by 
Lemma \ref{xxlem3.3}(a). By the proof of Lemma \ref{xxlem3.3}(a), 
$P_{x^{n_1p_1}, 1,1}\neq k(x^{n_1p_1}-1)$. Therefore $x^{n_1p_1}$ is
the major weight and $1$ is the major commutator. 

Since $H$ is not generated by $x^{\pm 1},y_1$ and since the major skew 
primitive elements of $H$ is generated by $y_1^{p_1}$, $(x^{n_1p_1}, 1)
\neq (x^{n_2},\lambda_2)$. By Lemma \ref{xxlem3.3}(b), $\lambda_2\in
\sqrt{\;}$. Say  $\lambda_2$ is a primitive $p_2$nd root of unity for 
some $p_2>1$. An argument similar to the above shows that 
\begin{enumerate}
\item[(i)]
$y_2 x=q_2 x y_2$ and $\lambda_2= q_2^{n_2}$;
\item[(ii)]
$x^{n_2p_2}$ is the major weight, and by the uniqueness of the major 
weight [Lemma \ref{xxlem3.3}(c)], we have $n_1p_1=n_2 p_2$; 
\item[(iii)]
$y_2^{p_2}\in P_{x^{n_2p_2}, 1,1}=P_{x^{n_1p_1}, 1,1}$, and whence
$y_2^{p_2}$ and $y_1^{p_1}$ are linearly dependent in 
$P'_{x^{n_2p_2}, 1,1}$. 
\end{enumerate}
Up to this point we have proved (a) and (d). 

Part (e) follows from part (d) and the fact  $\gcd(n_1,n_2)=1$.

(b,c) After replacing $y_1$ by a scalar multiple, (iii) implies that
\begin{equation}
\label{I5.2.1}\tag{I5.2.1}
y_1^{p_1}=y_2^{p_2}+\alpha(x^{M}-1)
\end{equation}
where $M=n_1p_1=n_2p_2$. Together with $y_i x=q_i xy_i$ for $i=1,2$, 
one derives that $q_1^{p_1}=q_2^{p_2}$ when $x$ is commuted with
relation \eqref{I5.2.1}. From this,
$$q_1^{n_2 p_1}=q_2^{n_2p_2}=1$$
and, by definition,
$$q_1^{n_1p_1}=1.$$ 
Since $\gcd(n_1,n_2)=1$, we obtain that $q_1^{p_1}=1$. Therefore 
both $q_1$ and $q_1^{n_1}$ are primitive $p_1$st roots of unity. 
Hence (b) holds. Note that (c) is similar.
As a consequence, $y_1^{p_1}$ and $x^M$ are central elements in $H$.

(f) By the end of proof of (b,c), $y_1^{p_1}$ and $x^M$ are central 
elements in $H$. To finish (f) note that the subalgebra generated by 
$y_1^{p_1}$ and $x^{\pm M}$ is a Hopf subalgebra of $H$, which is 
central in $H$ and is of GK-dimension two. The last assertion in (f) 
follows from \cite[Corollary 2]{SmZ}. 

(g) A result of Takeuchi \cite[Theorem 3.2]{Ta1} says that a Hopf 
algebra $H$ is faithfully flat over its Hopf subalgebra if the 
coradical of $H$ is cocommutative. Let $K_0=k[x^{\pm M}]$ and $K_1
=k[x^{\pm M},y_1^{p_1}]$. Since $y_1^{p_1}$ is a nontrivial
skew primitive element, $K_0\neq K_1$. By part (f) both $K_0$ and 
$K_1$ are two distinct central Hopf subalgebras of $H$. By 
\cite[Proposition 3.4.3]{Mo}, $HK_0^{+}\neq HK_1^{+}$. 
In particular, $y_1^{p_1}\not\in H(x^M-1)$. This means that 
$y_1^{p_1}$ is a nonzero primitive element in $H':=H/(x^M-1)$.

(h) Let $K$ be the Hopf subalgebra of $H$ generated by 
$y_1$ and $x^{\pm n_1}$. 
By the relations of $K$, one sees that $K$  is a quotient Hopf algebra 
of $A(1, q_1^{n_1})$ [Example \ref{xxex1.1}]. Since both $K$ and 
$A(1,q_1^{n_1})$ are domains of GK-dimension two, $K\cong 
A(1, q_1^{n_1})$. By part (f) $x^M$ is central in $H$ and in $K$. 
Let $K'=K/(x^M-1)$ and $H'=H/(x^M-1)$ and recycle the letters for 
the elements in $K'$ and $H'$. By part (f) $H$ is PI. Since any
affine  PI algebra is catenary, the principal ideal theorem
implies that $\GKdim H'=\GKdim H-1=1$. Using this fact and 
\eqref{I5.2.1}, one sees that $y_1$ is nonzero in $H'$. Then 
the equation $y_1x^{n_1}=q_1^{n_1}x^{-n_1}y_1$ implies that 
$y_1$ is a nontrivial skew primitive element in $H'$. 
By part (g), $y_1^{p_1}$ is a nonzero (and whence nontrivial) 
primitive element in $H'$. 

Let $\phi'$ denote the natural Hopf map $K'\to H'$ which maps 
$y_1\in K'$ to a nontrivial skew primitive element $y_1\in H'$. 
Since the only nontrivial skew primitive elements in $K'$ are 
generated by $y_1+k(x^{n_1}-1)$ and $y_1^{p_1}$, the assertion
proved in the last paragraph says that $\phi'$ is injective 
when restricted to $C_1(K')$. By \cite[Theorem 5.3.1]{Mo}, $\phi'$
is injective. By \cite[Theorem 3.2]{Ta1} $H$ (respectively, $H'$) 
is faithfully flat over $K$ (respectively, $K'$). Since $y_1^{p_1}$ 
is a nonzerodivisor of $K'$, it is also a nonzerodivisor of $H'$. Thus 
$$\GKdim H/(y_1^{p_1},x^M-1)\leq \GKdim 
H/(x^M-1)-1\leq \GKdim H-2=0.$$
Since $H$ (and hence $H(y_1^{p_1},x^M-1)$) is affine,
$H/(y_1^{p_1},x^M-1)$ is finite dimensional. We proved the first part 
of (h). 

For the second part of (h), note that $y_1$ is a nonzerodivisor of $K'$.
So it is a nonzerodivisor of $H'$ since $H'$ is faithfully flat over 
$K'$. If $y_1=0$ in $H/(y_1^{p_1},x^M-1)=H'/(y_1^{p_1})$, then 
$y_1=y_1^{p_1}f$ for some $f\in H'$. Re-writing it as 
$y_1(1-y_1^{p_{1}-1}f)=0$ yields a contradiction with that fact
$y_1$ is a nonzerodivisor of $H'$. Therefore $y_1$ is nonzero in
$H/(y_1^{p_1},x^M-1)$. By symmetry, $y_2$ is nonzero in
$H/(y_1^{p_1},x^M-1)$.

(i) The proof is similar to the proof of (b).
\end{proof}

The next result uses Proposition \ref{xxprop4.2}.

\begin{theorem}
\label{xxthm5.3} Retain the hypotheses of Theorem \ref{xxthm5.1}
and the notation of Lemma \ref{xxlem5.2}. Additionally assume that
$H$ is not isomorphic to any of $N_5,N_7,N_{10}$ and $N_{21}$. 
Let $q_{ij}=q_j^{n_i}$ for $i,j\in \{1,2\}$. Then $q_{12}q_{21}=1$.
\end{theorem}

\begin{proof} Let $A=H/(y_1^{p_1},x^M-1)$. This is a quotient 
Hopf algebra of $H$. By Lemma \ref{xxlem5.2}(h), $A$ is finite
dimensional. Also by Lemma \ref{xxlem5.2}(h), the image of 
$y_1$, which is still denoted by $y_1$, (respectively, the image 
of $y_2$) is nonzero in $A$. If $y_1$ and $y_2$
are linear dependent in $A$, then $n_1=n_2=1$ and $q_1=q_2$. This 
contradicts Theorem \ref{xxthm3.7} since $y_1$ and $y_2$
are linearly independent in $H/C_0$ by hypothesis (c) of
Theorem \ref{xxthm5.1}. Therefore these two 
are linearly independent nontrivial skew primitive elements of $A$.
Let $B$ be the associated graded Hopf algebra of $A$ with respect 
to its coradical filtration. Then $B=C\# G$ where 
$G={\mathbb Z}/(M)$ and $C$ is a braided Hopf algebra. Let $V$ be 
the subspace of $B$ (also viewed as a subspace of $C$) spanned by 
$y_1$ and $y_2$. Then $V$ is a Yetter-Drinfel'd module over $G$ 
(recall that $G={\mathbb Z}/(M)=\langle x\mid x^M=1\rangle$) with 
$$\delta(y_i)=x^{n^i}\otimes y_i, \quad x^{n_i}* y_j=q_j^{n_i} y_j$$
for all $i,j\in \{1,2\}$. The Nichols algebra over $V$, denoted by
${\mathcal B}(V)$, is a subquotient of $C$. Therefore 
${\mathcal B}(V)$ is finite dimensional over $k$. By Proposition 
\ref{xxprop4.2}, up to a permutation, one of cases (I)-(VI) holds.
We analyze these six case below. 

Case (I) is our assertion.

Case (II): $n_1=n_2=1$ and $q_1,q_2\in R_3$. Then either $q_2=q_1$ or 
$q_2=q_1^{-1}$. When $q_2=q_1$, it yields a contradiction with 
Theorem \ref{xxthm3.7}. Therefore $q_2=q_1^{-1}$, or $q_1 q_2=1$.
Hence the assertion. 

Case (III): $n_1=n_2=1$ and $q_1, q_2\in R_5$. Then $q_2=q_1^i$
for some $1\leq i\leq 4$. By Theorem 
\ref{xxthm3.7}, $q_2=q_1$ is impossible. The case $q_2=q_1^4$ is
our assertion. It remains to  study $q_2=q_1^2$ and $q_2=q_1^3$.
These two are equivalent since $q_2=q_1^2$ is equivalent to 
$q_1=q_2^3$, so we only consider the case when $q_2=q_1^2$. 
Under this condition, the algebra $H$ is $N_5$ in Definition
\ref{xxdefn4.4}(a). 

Case (IV): $n_1=1, n_2=2$, $\epsilon=5$, $p_1=10$, $p_2=5$ and
$q_1^4q_2=1$ and $q_1^2q_2^3=1$. This is the algebra $N_{10}$
in Definition \ref{xxdefn4.4}(b).

Case (V): $n_1=n_2=1$, $\epsilon=7$, $q_1, q_2\in R_7$, $q_1q_2^2=1$ and
$q_1^4 q_2=1$. This is the algebra $N_{7}$
in Definition \ref{xxdefn4.4}(c).

Case (VI): $n_1=1, n_2=3$, $\epsilon=7$, $p_1=21$, $p_2=7$, 
$q_1^3 q_2^4=1$ and that $q_1^6 q_2=1$. This is the algebra $N_{21}$
in Definition \ref{xxdefn4.4}(d). 

Combining all cases with the additional hypothesis, the assertion follows.
\end{proof}

\begin{lemma}
\label{xxlem5.4} Retain the hypotheses of Theorem \ref{xxthm5.1}
and the notation of Lemma \ref{xxlem5.2}. Additionally assume that
$H$ is not isomorphic to any of $N_5,N_7,N_{10}$ and $N_{21}$. Then
\begin{enumerate}
\item
$n_1\neq n_2$. 
\item
$y_2y_1-q_2^{n_1}y_1 y_2=0$. 
\end{enumerate}
\end{lemma}

\begin{proof} First we prove the following:
\begin{enumerate}
\item[(c)]
If $n_1=n_2=1$, then $p_1=p_2>2$.
\item[(d)]
$y_2y_1-q_2^{n_1}y_1 y_2$ is a skew primitive element in $C_0$. 
\end{enumerate}

(c) Clearly, $p_1=n_2\epsilon =\epsilon=n_1\epsilon=p_2$.
If $p_1=p_2=1$, then $q_1=q_2=1$ and $y_1$ and $y_2$ are both
major skew primitive, so $y_1$ and $y_2$ are not linearly
independent in $H/C_0$ by Theorem \ref{xxthm3.7}, which yields a 
contradiction by hypothesis (c) of Theorem \ref{xxthm5.1}. 

If $p_1=p_2=2$, then $q_1=q_2=-1$. By \eqref{I5.2.1},
$$y_1^2=y_2^2+a(x^2-1)$$
for some $a\in k$. Let $y_3= y_1y_2+y_2y_1$. Then $y_3$ is a skew 
primitive of weight $x^2$. (In general, if $q_1^{n_2}q_2^{n_1}=1$, then 
$y_1y_2-q_2^{n_1}y_2y_1$ is skew primitive of weight $x^{n_1+n_2}$). 
Therefore
$$y_1y_2+y_2y_1= b y_2^2+c (x^2-1)$$
for some $b, c\in k$.
Pick some scalars $\alpha, \beta, \gamma$ satisfying the equations
$$\begin{aligned} 
1+\alpha^2+\alpha b&=0\\
a+\alpha c+\beta^2&=0\\
-a-\alpha c-\gamma^2&=0.
\end{aligned}
$$
Then 
$$\begin{aligned}
(y_1+&\alpha y_2+\beta x +\gamma)(y_1+\alpha y_2+\beta x -\gamma)\\
=& y_1^2+\alpha^2 y_2^2 +\alpha(y_1y_2+y_2y_1)+\beta (y_1 x+x y_1)
+\alpha\beta (y_2 x+x y_2)+\beta^2 x^2\\
&+\gamma (y_1-y_1)+\alpha \gamma (y_2-y_1)+\beta\gamma(x-x)-\gamma^2\\
=&y_2^2+a(x^2-1)+\alpha^2 y_2^2+\alpha(b y_2^2+c (x^2-1))+\beta^2 x^2
-\gamma^2\\
=&0.
\end{aligned}
$$
Since $H$ is a domain, either $y_1+\alpha y_2+\beta x +\gamma=0$
or $y_1+\alpha y_2+\beta x -\gamma=0$. Both leads to a contradiction
with hypothesis (c) of Theorem \ref{xxthm5.1}. Therefore $p_1=p_2>2$.

(d) By Theorem \ref{xxthm5.3}, $q_1^{n_2}q_2^{n_1}=1$. Then 
$y_3:=y_2y_1-q_2^{n_1}y_1 y_2$ is a skew primitive by a direct computation. 
Suppose on the contrary that $y_3\not\in C_0$. It is easy to see that 
$\omega(y_3)=(x^{n_1+n_2},(q_1q_2)^{n_1+n_2})$. If $x^{n_1+n_2}$ is a major 
weight (or $y_3$ is a major skew primitive), then $n_1+n_2=n_1p_1=n_2p_2$. 
Since  $\gcd(n_1,n_2)=1$, we get $n_1=n_2=1$ and $p_1=p_2=2$. By 
part (c) this is impossible. Therefore $x^{n_1+n_2}$ is not the major 
weight. Consequently, we have that $n_1+n_2\neq n_1p_1=n_2p_2$ and that 
$y_3$ is not a major primitive element. Let $n_3=n_1+n_2$. By 
Lemma \ref{xxlem5.2}(i), 
$q_3:=q_1q_2$ and $\lambda_3=q_3^{n_3}$ are primitive $p_3$rd roots of 
unity and $n_3p_3=n_1p_1=n_2p_2$. Since $\gcd(n_1,n_2)=1$, 
$\gcd(n_3,n_1)=\gcd(n_1+n_2, n_1)=\gcd(n_2, n_1)=1$. This implies 
that $n_i\mid p_3$ for $i=1,2$. Since  $y_1$ and $y_3$ are non-major 
skew primitives of different weight, then the Hopf subalgebra $H'$ 
generated by $x^{\pm 1}, y_1, y_3$ satisfies the hypotheses in 
Theorem \ref{xxthm5.1}(a-d). 

Since $n_3=n_1+n_2>1$, the Hopf domain $H'$ is isomorphic to neither 
$N_5$ nor $N_7$. If $H'$ is isomorphic to $N_{10}$, then $n_3=2, 
n_1=1$, $p_1=10, p_3=5$ and $q_3=q_1^6$. Hence $q_2=q_3 q_1^{-1}=q_1^{5}$
has order $2$, contradicting $p_2=n_1p_1/n_2=p_1=10$.  If $H'$ is 
isomorphic to $N_{21}$, then $n_3=3, n_1=1$, $p_1=21, p_3=7$ and 
$q_3=q_1^{15}$. Hence $n_2=2$, contradicting $n_2p_2=n_1p_1=21$. 
Therefore the additional hypothesis in Theorem \ref{xxthm5.3} holds for 
$H'$. Applying Theorem \ref{xxthm5.3} to $H'$ we 
obtain that $q_1^{n_3}q_3^{n_1}=1$. Now
$$q_1^{2n_1}=q_1^{2n_1} (q_1^{n_2}q_2^{n_1})=q_1^{n_1+n_2}(q_1q_2)^{n_1}=
q_1^{n_3} q_3^{n_1}=1.$$
Since $q_1^{n_1}$ is a $p_1$st primitive root of unity, $p_1=2$. 
Since $p_1=n_3 \epsilon_3$ by Lemma \ref{xxlem5.2}(e) for $H'$,
$\epsilon_3=1$. This implies that $p_1=n_3=2$ and $p_3=n_1$ by
Lemma \ref{xxlem5.2}(e). Since $n_1+n_2=n_3=2$, we have $n_1=n_2=1$
and $q_1^{2}=1$ and $q_3^{1}=1$. This means that $y_3$ is a major
skew primitive element, a contradiction. 

Now we go back to prove the lemma.

(a) Suppose $n_1=n_2$. Then $n_1=n_2=1$ as $\gcd(n_1,n_2)=1$ and 
$q_1q_2=1$ by Theorem \ref{xxthm5.3}. By part (c) $p:=p_1=p_2>2$. 
By part (d), $y_2y_1-q_2 y_1y_2$ is a skew primitive in $C_0$.
Hence 
$$y_2y_1-q_2 y_1y_2=b(x^2-1)$$
for some $b\in k$. 
By \eqref{I5.2.1}, we have
$$y_1^p=y_2^p+a(x^{p}-1).$$
Pick $\alpha$ and $\beta$ such that
$$\alpha\beta (1-q_2)=-b, \quad \beta^p-\alpha^p=a.$$
Then 
$$(y_1+\alpha x)^p=(y_2+\beta x)^p-a$$
and
$$(y_2+\beta x)(y_1+\alpha x)-q_2(y_1+\alpha x)(y_2+\beta x)
=-b.$$
Thus the subalgebra $Y$ generated by $y_1+\alpha x$ and $y_2+\beta x$
has GK-dimension at most one. Since $Y$ is a domain, it is commutative by
\cite[Lemma 4.5]{GZ}. So $(y_2+\beta x)(y_1+\alpha x)=(q_2-1)^{-1}b$
or
$$y_2y_1+\alpha q_2 xy_2+\beta x y_1+\alpha \beta x^2=c$$
where $c=(q_1-1)^{-1}b$. After applying $\Delta$, we see that
$\{y_1,y_2,x,1\}$ are linearly dependent. 
This contradicts hypothesis (c) of Theorem \ref{xxthm5.1}.

(b) By part (d) $y_2 y_1-q_2^{n_1} y_1y_2$ is a skew 
primitive of weight $x^{n_1+n_2}$ and in $C_0$. So we have
$$y_2y_1-q_2^{n_1}y_1y_2=a(x^{n_1+n_2}-1)$$
for some $a\in k$. If $a\neq 0$, by commuting with $x$, the above 
equation implies that $q_1 q_2=1$. By part (a), 
$n_1\neq n_2$, by symmetry, we may assume  $n_2>n_1\geq 1$.
Then together with $q_1^{n_2}q_2^{n_1}=1$ we have $q_1^{n_2-n_1}=1$.
Since $q_1$ is a $p_1$st root of unity and $p_1=n_2\epsilon$, 
$n_2\epsilon$ divides $n_2-n_1$. Since $\gcd(n_1,n_2)=1$, we obtain
$n_2=1$, a contradiction. Therefore $a=0$ and the assertion follows.
\end{proof}

We are now ready to prove Theorem \ref{xxthm5.1}.

\begin{proof}[Proof of Theorem \ref{xxthm5.1}] Suppose $H$ is not
isomorphic to any of $N_5,N_7,N_{10}$ and $N_{21}$. First we claim 
that $q_1^{n_2}=q_2^{n_1}=1$. By Lemma \ref{xxlem5.4}(b),
$$y_2y_1-q_2^{n_1}y_1y_2=0.$$
Now we see that all relations of $K:=K(\{p_1,p_2\},\{q_1,q_2\},
\{\alpha_1,\alpha_2\}, M)$ (where $M=p_1n_1$) as listed (I2.1.1)-(I2.1.6) 
are satisfied by $H$. Then $H$ is isomorphic to a quotient Hopf 
algebra of $K$. Since $H$ is a domain, by Lemma \ref{xxlem2.3}(b), 
$q_j^{n_i}=1$.

Since $q_j^{n_i}=1$ for all $i\neq j$, $p_j\mid n_i$ for $i\neq j$. 
Thus $\gcd(p_1,p_2)=\gcd(n_2,n_1)=1$. Hence 
$K$ is in fact the algebra $B(1,\{p_i\}_{1}^2,q,\{\alpha_i\}_{1}^2)$ 
by Proposition \ref{xxprop2.4} and Convention \ref{xxcon2.5}. 
So we have a surjective algebra map from 
$B(1,\{p_i\}_{1}^2,q,\{\alpha_i\}_{1}^2)$ 
to $H$ between two Hopf domains of GK-dimension two. This map must 
be an isomorphism.
\end{proof}

\section{Proof of Theorem \ref{xxthm0.1} and Corollary \ref{xxcor0.2}}
\label{xxsec6}

In this final section we put together a proof of Theorem
\ref{xxthm0.1} and Corollary \ref{xxcor0.2}.

\begin{definition}
\label{xxdefn6.1} A Hopf algebra $H$ is said to satisfy the hypothesis
$\Omega$ if it does not contain any of $N_5, N_7, 
N_{10}$ and $N_{21}$ as a Hopf subalgebra.
\end{definition}

\begin{theorem}
\label{xxthm6.2}
Let $H$ be a pointed Hopf domain of GK-dimension strictly less 
than three. Suppose that $H$ is finitely generated by grouplike and
skew primitive elements and that $H$ satisfies $\Omega$. If the 
coradical $C_0$ has GK-dimension one, then $H$ is isomorphic to 
either $k{\mathbb Z}$, or one of algebras in Theorem 
\ref{xxthm1.4}(III,IV,V) or the algebra 
$B(n,\{p_i\}_{1}^s,q,\{\alpha_i\}_{1}^s)$. 
\end{theorem}

\begin{proof} If $H$ is not PI, the assertion follows from 
Corollary \ref{xxcor1.12}. So we may assume $H$ is  PI. Then 
$\GKdim H$ is an integer, either 1 or 2. If $\GKdim H=1$,
the assertion follows from \cite[Proposition 2.1]{GZ} together with 
Lemmas \ref{xxlem1.6} and \ref{xxlem1.10}. It remains to consider
the case $\GKdim H=2$ and $H$ is PI. By Lemma \ref{xxlem1.10},
$C_0$ is affine. Since $C_0=k\Gamma$ where $\Gamma$
is an abelian torsionfree group rank 1, $\Gamma$ is isomorphic 
to ${\mathbb Z}$. Let $x$ be a generator of $C_0$. By Lemma 
\ref{xxlem3.3}(c), $\dim P'_{M}=1$. Note that $H$ is generated by
$C_0$ and preimages of $P'_{T}$. 

Let $s=\dim P'_{T}-\dim P'_{M}$. By Theorem \ref{xxthm3.9},
$s$ is finite. 

If $s=0$, then only nontrivial
skew primitive element $z$ is in $P'_{M}$. So $H$ is generated by
$x^{\pm 1},z$. Suppose $\omega(z)=(x^M, \lambda)$ and without loss of 
generality $M\geq 0$. If $\lambda\neq 1$, then 
we may further assume $x^{-1} z x=q z$ by Lemma \ref{xxlem3.8}
and $\lambda=q^M$. Therefore there is a Hopf surjective map $A(M,q)
\to H$. This is an isomorphism since both algebras are domains of
GK-dimension two. If $\lambda=1$, then either $x^{-1} z x=qz$ or 
$x^{-1} z x=z+(x^M-1)$.  We have already dealt with the first case, 
and the second case forces $H$ to isomorphic to $C(M-1)$ defined 
in Example \ref{xxex1.3}, which is non-PI. 

Note that when $s=0$, $H$ is non-PI. Hence, under the
hypothesis that $H$ is PI (see the first paragraph of
the proof), $s>0$. 

If $s=1$, let $y$ be a nontrivial non-major skew primitive
with $\omega(y)=(x^n, \lambda)$ for some $n>0$. Then $\lambda$
is a $p$th primitive root of unity for some $p\geq 2$ and $y^p$ 
is a  nontrivial major skew primitive. So $H$ is generated by 
$x^{\pm 1}$ and $y$. An argument similar to above shows that $H\cong 
A(n,q)$ for some $q\in k^{\times}$. 

Suppose now $s\geq 2$. Pick any two linearly independent nontrivial 
non-major skew primitive elements, say $y_1, y_2$, with $\omega(y_i)
=(x^{n_i}, \lambda_i)$. Let $X=x^{n}$ where $n=\gcd(n_1,n_2)$ and
let $K$ be the Hopf subalgebra generated by $X^{\pm 1}$, $y_1$ and 
$y_2$. Since $\lambda_i\neq 1$, $K$ is noncommutative. Therefore
$\GKdim K=2$ by \cite[Lemma 4.5]{GZ}. By Lemma \ref{xxlem1.10}
and its proof, the coradical $C_0(K)$ is generated by $X^{\pm 1}$. In
$K$, $\omega(y_i)=(X^{n'_i}, \lambda_i)$ where $n'_i=n_i/n$. Thus
$\gcd(n'_1,n'_2)=1$. Therefore $K$ satisfies all hypotheses in
Theorem \ref{xxthm5.1}. As a consequence, $y_1y_2=y_2y_1$. 
Other consequences are
\begin{enumerate}
\item[(I6.2.1)]
any major skew primitive element is generated by $C_0$ and $y_1$,
\item[(I6.2.2)]
both $n_1$ and $n_2$ are positive,
\item[(I6.2.3)]
$n_1p_1=n_2p_2$ and $\gcd(p_1,p_2)=1$ where $p_i$ is the order of $\lambda_i$, 
\item[(I6.2.3)]
$y_2^{p_2}=y_1^{p_1}+\alpha_2(x^M-1)$, where 
$y_1^{p_1}$ is a major skew primitive of weight $x^M$. 
\end{enumerate}

Re-cycling the notations $n_i$ etc. For each nontrivial 
non-major $P_{x^{n_i}, \lambda_i,*}$, there is a 
$y_i\in P_{x^{n_i}, \lambda_i,*}\setminus C_0$, unique up to a scalar
multiple by Lemma \ref{xxlem3.8}(a), such that 
\begin{equation}
\label{I6.2.4}\tag{I6.2.4}
x^{-1} y_ix=q_i y_i.
\end{equation}
Then $\lambda_i=q_i^{n_i}$. By (I6.2.1), $H$ is generated by
$x^{\pm 1}, y_1,\cdots, y_s$, and
\begin{enumerate}
\item
every $n_i$ is positive,
\item
$n_1p_1=n_2p_2=\cdots=n_sp_s=:M$ where $p_i$ is the order of $\lambda_i$, 
\item
$y_i^{p_i}=y_1^{p_1}+\alpha_i(x^M-1)$, where $\alpha_i\in k$.
\end{enumerate}
By commuting $x$ with the equation in (c) above, one sees that 
$q_1^{p_1}=q_i^{p_i}$. Note that 
$\Delta(y_i)=y_i\otimes 1+x^{n_i}\otimes y_i$. Since 
$y_iy_j=y_jy_i$, expanding $\Delta(y_iy_j)=\Delta(y_jy_i)$ shows
that $q_i^{n_j}=q_j^{n_i}=1$. 

If all $\alpha_i=0$, consider the subalgebra $Y$ generated by $y_i$, 
which is a commutative algebra 
$k[y_1,\cdots, y_s]/(y_i^{p_i}=y_1^{p_1}\mid i\geq 2)$. By the proof of 
\cite[Construction 1.2]{GZ}, $Y$ is isomorphic to a subalgebra of
$k[y,y^{-1}]$ by identifying $y_i$ with $y^{m_i}$.  Similar to
the proof of Proposition \ref{xxprop2.4}, one see that $q_i=q^{m_i}$
and that $H=Y[x,x^{-1};\sigma]$ for some graded algebra automorphism
$\sigma$ of $Y$. As a consequence, $H/[H,H]=k[x^{\pm 1}]$. By
\cite[Theorem 3.8(c)]{GZ}, \nat\; holds. By Theorem \ref{xxthm1.4}, 
$H$ is isomorphic to the algebra $B(n,p_0,\dots,p_s,q)$ defined 
in Example \ref{xxex1.2}.

If some $\alpha_i\neq 0$ (and assume $\alpha_1=0$), then $q_1^{p_1}=q_i^{p_i}
=1$ for all $i$. Thus $q_i$ is also a $p_i$th primitive root of unity.
Then all conditions (I2.0.1)-(I2.0.8) are verified and all relations
(I2.0.9)-(I2.0.12) hold in $H$. By Proposition \ref{xxprop2.4} and Convention 
\ref{xxcon2.5}, the algebra with relations (I2.0.9)-(I2.0.12) is 
$B(n,\{p_i\}_{1}^s,q,\{\alpha_i\}_{1}^s)$. Thus there is a Hopf algebra 
surjective map from $B(n,\{p_i\}_{1}^s,q,\{\alpha_i\}_{1}^s)\to H$ which 
must be an isomorphism since both algebras are domains of GK-dimension 
two. Therefore the assertion follows. 
\end{proof}

\begin{proof}[Proof of Theorem \ref{xxthm0.1} and 
Corollary \ref{xxcor0.2}] 
Since all algebras in Theorem \ref{xxthm1.4} satisfy the condition
$\Ext^1_H(k,k) \neq 0$, Theorem \ref{xxthm0.1} and 
Corollary \ref{xxcor0.2} are equivalent. We will basically
prove Corollary \ref{xxcor0.2}. 

As the discussion given in subsection \ref{xxsec1.2}, $\GKdim C_0$
is either 0, 1, or 2. If $\GKdim C_0=0$, the assertion follows
from Theorem \ref{xxthm1.9}. If $\GKdim C_0=2$, then the assertion 
follows from Theorem \ref{xxthm1.7}. 

It remains to deal with the case that $\GKdim C_0=1$. It is clear 
that the hypothesis $\Omega$ follows from the hypothesis $\Omega'$. 
Hence we may apply Theorem \ref{xxthm6.2}. By Theorem \ref{xxthm6.2}
and by the fact $\GKdim H=2$, $H$ is isomorphic to one of algebras 
in Theorem \ref{xxthm1.4}(III,IV,V) or the algebra 
$B(n,\{p_i\}_{1}^s,q,\{\alpha_i\}_{1}^s)$. When $H$ is isomorphic to
$B(n,\{p_i\}_{1}^s,q,\{\alpha_i\}_{1}^s)$ and $\Ext^1_H(k,k)=0$, 
Lemma \ref{xxlem2.3}(c) says that $\alpha_i\neq \alpha_j$ for some 
$i$ and $j$. This finishes the proof. 
\end{proof}

\providecommand{\bysame}{\leavevmode\hbox to3em{\hrulefill}\thinspace}
\providecommand{\MR}{\relax\ifhmode\unskip\space\fi MR }
\providecommand{\MRhref}[2]{%

\href{http://www.ams.org/mathscinet-getitem?mr=#1}{#2} }
\providecommand{\href}[2]{#2}

\end{document}